\documentclass[11pt]{amsart}
\pdfoutput=1

\usepackage[usenames,dvipsnames]{xcolor}   %

\usepackage[utf8]{inputenc}
\usepackage[T1]{fontenc}

\usepackage[letterpaper,margin=1in]{geometry}

\usepackage{lmodern}
\usepackage{microtype}
\usepackage{etoolbox}                 %
\usepackage{lettrine}

\usepackage{mathtools}
\usepackage{amsmath,amssymb,amsfonts,amsxtra,bm}
\usepackage{amsthm}
\usepackage{mathrsfs}
\usepackage{latexsym}
\usepackage{nicefrac}       %

\usepackage{booktabs}
\usepackage{makecell,multirow}       %
\usepackage{colortbl}
\usepackage{graphicx}
\usepackage{epstopdf}
\usepackage{float}
\floatstyle{plaintop}
\restylefloat{table}
\usepackage{footnote, tablefootnote}

\usepackage{enumitem}
\usepackage{algorithm}
\usepackage{algorithmic}
\usepackage{comment}
\usepackage{xspace}
\usepackage{url}
\usepackage[most]{tcolorbox}

\usepackage[round]{natbib}

\usepackage{hyperref}

\numberwithin{equation}{section}

\makeatletter

\definecolor{darkblue}{rgb}{0.0,0.0,0.65}
\definecolor{darkred}{rgb}{0.68,0.05,0.0}
\definecolor{darkgreen}{rgb}{0.0,0.29,0.29}
\definecolor{darkpurple}{rgb}{0.47,0.09,0.29}

\patchcmd{\section}{\scshape}{\fontfamily{ppl}\bfseries\scshape\selectfont}{}{%
  \PackageWarning{aiproofs}{Could not bold the section font; amsart may have changed}}

\newcommand{\aiauthor}{%
  \author[Suvrit Sra]{Suvrit Sra\textsuperscript{\normalfont*}}%
  \thanks{*\,TU Munich, Dept.~of Mathematics, School of CIT; Garching, Germany.
    Email: \texttt{s.sra@tum.de}.}}

\makeatother

\hypersetup{colorlinks=true,linkcolor=blue!45!black,citecolor=blue!45!black,urlcolor=red!65!black}
\setlist{nosep,leftmargin=*}

\newenvironment{aiusage}{\section*{AI usage}}{}

\newcommand{\aiacknowledgements}[1][]{%
  \section*{Acknowledgements}
  \if\relax\detokenize{#1}\relax\else #1\ \fi
  The author (SS) thanks the Alexander von Humboldt (AvH) Foundation for
  their generous support via a Humboldt Professorship in AI.%
}

\newcommand{\reals}{\mathbb{R}}

\newcommand{\pp}{\mathbb{P}}

\newcommand{\Fc}{\mathcal{F}}

\newcommand{\half}{\tfrac{1}{2}}

\newcommand{\ip}[2]{\langle {#1},\, {#2} \rangle}

\newcommand{\set}[1]{\left\lbrace#1\right\rbrace}

\newcommand{\norm}[1]{\|{#1}\|}

\newcommand{\pnorm}[2]{\|#1\|_{#2}}

\newcommand{\frob}[1]{\|#1\|_{\text{F}}}

\newcommand{\what}{\widehat}

\DeclareMathOperator{\trace}{tr}
\DeclareMathOperator{\diag}{diag}

\DeclareMathOperator{\real}{Re}

\newtheorem{theorem}{Theorem}[section]

\newtheorem{prop}[theorem]{Proposition}
\newtheorem{lemma}[theorem]{Lemma}
\newtheorem{corr}[theorem]{Corollary}
\newtheorem{rmk}[theorem]{Remark}
\newtheorem{example}[theorem]{Example}

\theoremstyle{definition}

\definecolor{cdarkred}{RGB}{180,70,30} %
\definecolor{cdarkgreen}{RGB}{0,130,0} %
\definecolor{azure}{RGB}{240,255,255}
\definecolor{cmidgreen}{RGB}{201,219,216}
\definecolor{clightgreen}{RGB}{200,255,200}
\definecolor{clightgray}{RGB}{240,240,240}
\definecolor{cmygreen}{RGB}{10,120,0}
\definecolor{cmyblue}{RGB}{50,50,255}
\definecolor{sblue}{RGB}{100,100,255}
\definecolor{cmypurple}{RGB}{150,50,180}
\definecolor{cmyred}{RGB}{238, 64,  0}
\definecolor{cdarkblue}{RGB}{30,75,170}
\definecolor{csharpred}{RGB}{255,0,0}

\newtcolorbox{dashedbox}[1][]{
  nobeforeafter,
  colback=yellow!10!white,
  boxrule=0pt,
  enhanced jigsaw,
  top=2pt,
  left=2pt,
  right=2pt,
  bottom=0pt,
  borderline horizontal={0.5pt}{0pt}{dashed},
  borderline vertical={0.5pt}{0pt}{dashed},
  #1
}

\newtcbox{\tcbmathbox}[1][]{
  nobeforeafter,
  colback=yellow!30!white,
  boxrule=0pt,
  enhanced jigsaw,
  top=0pt,
  left=1pt,
  right=1pt,
  bottom=0pt,
  borderline horizontal={0.5pt}{0pt}{dashed},
  borderline vertical={0.5pt}{0pt}{dashed},
  #1
}

\newcommand{\parsum}{\mathbin{:}}      %
\newcommand{\sphere}{\mathbb{S}_{p,+}} %
\newcommand{\simplex}{\Delta_m}        %
\newcommand{\Bc}{\mathcal{B}}
\newcommand{\Kc}{\mathcal{K}}
\newcommand{\Lc}{\mathcal{L}}
\newcommand{\Rc}{\mathcal{R}}
\newcommand{\Tc}{\mathcal{T}}
\DeclareMathOperator{\rank}{rank}
\DeclareMathOperator{\ran}{ran}
\DeclareRobustCommand{\pbsubtitle}{{\fontfamily{ppl}\selectfont
  (with application to $p$-Bures--Wasserstein barycenters)}}

\usepackage{pgfplots}
\usepackage{placeins}
\pgfplotsset{compat=1.18}

\allowdisplaybreaks

\keywords{Bures--Wasserstein distance, Schatten norm, positive semidefinite block matrices, Lin--van den Driessche conjecture, partition inequality, complex interpolation, barycenter, Fr\'echet mean}

\hypersetup{
  pdftitle={The p-Bures--Wasserstein Distance via 3x3 Block Matrices, with
            Application to p-Bures--Wasserstein Barycenters},
  pdfauthor={Suvrit Sra},
  pdfsubject={A metric interpolating between the Bures--Wasserstein and
              Hilbert--Schmidt distances, and its Frechet means}
}

\begin{document}
\title[The $p$-Bures--Wasserstein distance]{The $p$-Bures--Wasserstein distance via $3\times3$ block matrices\\[0.5ex]
\pbsubtitle}
\aiauthor
\begin{abstract}
  We study the \(p\)-Bures--Wasserstein distance
  \(D_p(A,B)=(\trace A^p+\trace B^p-2\pnorm{A^{1/2}B^{1/2}}{p}^p)^{\frac12}\) on positive semidefinite
  matrices; it equals the Bures--Wasserstein distance for \(p=1\), and the Hilbert--Schmidt distance for
  \(p=2\).  We prove that \(D_p\) is a metric for every \(1\le p\le2\), by proving a stronger 2014
  conjecture of Lin and van den Driessche on \(3\times3\) block matrices.  Our main tool is a partition
  inequality for Schatten norms, proved by complex interpolation.  We also characterize Fr\'echet means
  for $D_p$ by a fixed-point equation and prove geometric convergence of the associated iteration.
\end{abstract}
\maketitle

\vspace*{-10pt}
\section{Introduction}\label{sec:intro}
Let \(\pp_n\) denote the cone of \(n\times n\) complex positive semidefinite
matrices, and write \(\pnorm{\cdot}{p}\) for the Schatten \(p\)-norm.  Fix $1\le p\le2$ and write $t:=\frac p2\in\left[\frac12,1\right]$. For \(A,B\in\pp_n\), define the \emph{affinity}
\begin{equation}\label{eq:affinity-definition}
  \Phi_p(A,B):=\trace\!\bigl(A^{1/2}BA^{1/2}\bigr)^t
  =\pnorm{A^{1/2}B^{1/2}}{p}^p.
\end{equation}
We study the \emph{\(p\)-Bures--Wasserstein distance}
\begin{equation}\label{eq:distance-definition}
  D_p(A,B):=\left(\trace A^p+\trace B^p-2\Phi_p(A,B)\right)^{1/2}.
\end{equation}
At \(p=1\), this is the Bures--Wasserstein distance
\citep{bures1969extension,olkin1982distance,bhatia2019bures}, while at \(p=2\)
it is the Hilbert--Schmidt distance.  The main result interpolates the triangle
inequality between these two endpoints.

The \(p\)th root of \(\Phi_p\), \(\pnorm{A^{1/2}B^{1/2}}{p}\), is the \(p\)-norm fidelity of
\citet{wei2024antilinear-pbures}, who record its elementary properties.
At \(p=1\) it is the root fidelity of quantum information.  Other
generalizations include the alpha Procrustes distances of
\citet{minh2022alpha}, which are based on matrix powers and approach the
Log-Euclidean distance, and the fidelities of
\citet{afham2025riemannian}, built from Bures--Wasserstein Riemannian geometry.
The family \(D_p\) varies the Schatten exponent.  Its triangle inequality is
implicit in a conjecture of \citet{lin2014positive} on positive
semidefinite \(3\times3\) block matrices, which we call the \emph{LvD
conjecture} (Section~\ref{sec:lin-block-matrix}).

The same family reappeared later in work on matrix divergences, without
reference to the LvD conjecture: for \(1\le p<2\), \(D_p(A,B)^2\) is twice the
\(\alpha\)-\(z\) Bures--Wasserstein divergence of \citet{dinh2021alphaz} at
\(\alpha=1/2\) and \(z=p/2\), evaluated at \((A^p,B^p)\), and
\citet[Problem~2]{komalovics2024parametric} asked whether the resulting
distance is a metric for \(1<p<2\).

\begin{theorem}[$p$-Bures--Wasserstein metric]\label{thm:main}
For every \(1\le p\le2\) and all \(A,B,C\in\pp_n\),
\begin{equation}\label{eq:main-triangle}
  D_p(A,C)\le D_p(A,B)+D_p(B,C).
\end{equation}
\end{theorem}
To control a weighted sum of affinities, we use the following inequality for
Schatten norms.
\begin{theorem}[Partition inequality]\label{thm:partition-intro}\label{thm:partition}
Let \(1\le p\le2\), let \(H_1,\dots,H_m\succeq0\), and suppose $H:=\sum_{j=1}^mH_j\succ0$.
Then, for every matrix \(X\),
\begin{equation}\label{eq:partition-intro}
  \sum_{j=1}^m\pnorm{H_j^{1/p}H^{-1/p}X}{p}^p\ge\pnorm{X}{p}^p.
\end{equation}
\end{theorem}
At \(p=2\), inequality \eqref{eq:partition-intro} is an
identity.  At \(p=1\) it follows from the triangle inequality. %
For intermediate \(p\), it follows by interpolating a row contraction and then
using a left-inverse relation.

\medskip\noindent\emph{\textbf{LvD conjecture.}}
We prove the triangle inequality behind Theorem~\ref{thm:main} for
angles on the normalized sphere (Theorem~\ref{thm:angular-triangle}); in
that form it resolves Conjecture 5.1 of \citet{lin2014positive}.

\begin{theorem}[LvD conjecture]\label{thm:lin-block-norm}
Let \(\mathbf H=[H_{ij}]_{i,j=1}^3\succeq0\) be a positive semidefinite
\(3\times3\) block matrix, with arbitrary finite block sizes compatible with
the block structure.  Then, for every \(1\le p\le2\),
\begin{equation}\label{eq:lin-block-norm-psd}
  \bigl[\pnorm{H_{ij}}{p}^p\bigr]_{i,j=1}^3\succeq0.
\end{equation}
\end{theorem}

The endpoints \(p=1\) and \(p=2\) were known
\citep{drury2014positive,marcus1971partitioned}.  For semidefinite block input,
the conclusion can be singular (Example~\ref{ex:lin-tight}).

\citet{vuong2026three} have also proved Theorem~\ref{thm:lin-block-norm}
and the metric property in Theorem~\ref{thm:main}, for finite tracial von
Neumann algebras, by a different method.\footnote{Our proofs of
Theorems~\ref{thm:main} and~\ref{thm:lin-block-norm} were obtained
independently of \citet{vuong2026three}, which appeared on arXiv on
8~October 2026.  They are recorded, together with the barycenter results of
Section~\ref{sec:barycenters}, in commits \texttt{02d4b2b5f5ae} and
\texttt{8baca98a9389} of the author's repository, pushed on 12~August 2026
at 10:08 and 21:02~UTC.}  Their proof uses tracial Carlen--Lieb convexity
to control the associated \(3\times3\) quadratic form.  Our partition
inequality gives an angular proof and also the sufficient decrease estimate
for the barycenter iteration (Lemma~\ref{lem:sufficient-decrease}).

\subsection{Barycenters}
The metric \(D_p\) defines Fr\'echet means
\citep{frechet1948elements}.  Given \(A_1,\dots,A_m\in\pp_n\) and weights
\(w\) in the probability simplex \(\simplex\), a \emph{barycenter} is a
minimizer of \(\Fc_w(X)=\sum_iw_iD_p(X,A_i)^2\) over \(\pp_n\).  The functional
\(\Fc_w\) is coercive and convex in the cone's linear structure,
and strictly convex for \(p>1\).  Barycenters exist throughout the range,
with uniqueness for \(p>1\)
(Theorem~\ref{thm:existence}).  Positive definite barycenters are
characterized by a fixed-point equation.

\begin{theorem}[Fixed-point characterization]\label{thm:euler}
Let \(1\le p\le2\), \(A_1,\dots,A_m\in\pp_n\), \(w\in\simplex\), and let
\(X\succ0\).  Then \(X\) is a barycenter if and only if
\begin{equation}\label{eq:euler}
  X^p=\sum_{i=1}^mw_i\bigl(X^{1/2}A_iX^{1/2}\bigr)^{p/2}.
\end{equation}
\end{theorem}

At \(p=2\) the equation reduces to \(X=\sum_iw_iA_i\).  At \(p=1\), for real
symmetric data, it characterizes the barycenter of the Gaussians \(N(0,A_i)\) in
\(2\)-Wasserstein space \citep{agueh2011barycenters,alvarez2016fixed}
(Corollaries~\ref{cor:p2} and~\ref{cor:p1}).  For positive definite data and
\(1\le p<2\), writing the equation in the variables \(X^p\) and \(A_i^p\)
recovers the characterization of the minimizer of a weighted sum of the
\(\alpha\)-\(z\) Bures--Wasserstein divergences above
\citep{dinh2021alphaz,jeong2023right}.

The equation suggests the
iteration \(X\mapsto\Tc(X)\) of \eqref{eq:fixed-point-map}.  For positive
definite data and positive weights, one step of this iteration decreases \(\Fc_w\) by at least the
square of the distance it moves (Lemma~\ref{lem:sufficient-decrease}), and the
iteration converges linearly from every positive definite start
(Theorem~\ref{thm:linear-convergence}, and
Proposition~\ref{prop:bures-endpoint-rate} at \(p=1\)).  An \(m\)-point form of
the support inequality below also gives a lower bound for the minimum value
of \(\Fc_w\) in terms of the pairwise distances of the data
(Theorem~\ref{thm:variance}).

\subsection{Notation and elementary facts}\label{sec:notation}

All matrices are finite-dimensional.  We write \(X\succ0\) and \(X\succeq0\)
for positive definiteness and semidefiniteness, take matrix powers as the
principal powers on the support, and approximate by \(X+\varepsilon I\) when
inverses appear.  For \(A\in\pp_n\), the radial coordinate is
\(r_p(A):=(\trace A^p)^{1/2}\).  The H\"older
inequality~\citep{simon2005trace,bhatia1997matrix} and
\eqref{eq:affinity-definition} give
\begin{equation}\label{eq:holder-affinity}
  0\le\Phi_p(A,B)\le\pnorm{A^{1/2}}{2p}^p\pnorm{B^{1/2}}{2p}^p=r_p(A)r_p(B),
\end{equation}
so \(D_p(A,B)^2\ge(r_p(A)-r_p(B))^2\).  On the normalized positive sphere
\(\sphere:=\set{A\in\pp_n:\trace A^p=1}\), homogeneity gives, for nonzero
\(A\) and \(B\),
\begin{equation}\label{eq:normalization}
  \what A:=r_p(A)^{-2/p}A\in\sphere,
  \qquad
  \Phi_p(A,B)=r_p(A)r_p(B)\Phi_p(\what A,\what B).
\end{equation}

\section{The Schatten partition inequality}\label{sec:partition}

The partition inequality follows once the column map below has a contractive
left inverse.  Write \(p'\) for the conjugate
exponent, so \(1/p+1/p'=1\).

\begin{proof}[\textbf{Proof of Theorem~\ref{thm:partition-intro}}]
For \(1<p\le2\) and \(H_j\succ0\) for every \(j\), define the column and row maps
\begin{align*}
  \Lc_p(X)&:=\bigl(H_j^{1/p}H^{-1/p}X\bigr)_{j=1}^m, &
  \Rc_p((X_j)_{j=1}^m)&:=\sum_{j=1}^mH^{-1/p'}H_j^{1/p'}X_j.
\end{align*}
They satisfy the exact left-inverse identity
\begin{align}
  \Rc_p\Lc_p(X)
  &=\sum_jH^{-1/p'}H_j^{1/p'}H_j^{1/p}H^{-1/p}X\notag\\
  &=H^{-1/p'}\left(\sum_jH_j\right)H^{-1/p}X=X.
  \label{eq:left-inverse}
\end{align}
It therefore remains to show that
\begin{equation}\label{eq:row-contraction-p}
  \pnorm{\Rc_p((X_j))}{p}
  \le
  \left(\sum_j\pnorm{X_j}{p}^p\right)^{1/p}.
\end{equation}

Consider the analytic family
\(\Rc_z((X_j)):=\sum_jH^{-z}H_j^zX_j\) on the strip
\(0\le\real z\le1/2\).
On the boundary \(z=is\), all imaginary powers are unitary, so
\begin{equation}\label{eq:l1-boundary}
  \pnorm{\Rc_{is}((X_j))}{1}
  \le\sum_j\pnorm{X_j}{1}.
\end{equation}
On the boundary \(z=1/2+is\), put \(V_j:=H^{-1/2-is}H_j^{1/2+is}\).  Then
\begin{equation}\label{eq:row-coisometry}
  \sum_jV_jV_j^*
  =H^{-1/2-is}\left(\sum_jH_j\right)H^{-1/2+is}=I.
\end{equation}
Thus \([V_1\ \cdots\ V_m]\) has operator norm one, and its
Hilbert--Schmidt amplification gives
\begin{equation}\label{eq:l2-boundary}
  \pnorm{\Rc_{1/2+is}((X_j))}{2}^2
  \le\sum_j\pnorm{X_j}{2}^2.
\end{equation}

Interpolate \eqref{eq:l1-boundary} and \eqref{eq:l2-boundary} at
\(\theta=2/p'\), so that \(1/p=1-\theta/2\) and \(\theta/2=1/p'\).
The complex interpolation identities for finite Schatten classes and their
\(\ell_q\) sums \citep{bergh1976interpolation,simon2005trace} give
\(\Rc_{1/p'}:\ell_p^m(S_p)\longrightarrow S_p\) with norm at most one.
Since \(\Rc_{1/p'}=\Rc_p\), this proves \eqref{eq:row-contraction-p}.
Combining it with \eqref{eq:left-inverse},
\[
  \pnorm{X}{p}
  =\pnorm{\Rc_p\Lc_p(X)}{p}
  \le\pnorm{\Lc_p(X)}{\ell_p(S_p)}.
\]
Taking the \(p\)th power proves \eqref{eq:partition-intro}.

At \(p=1\), the identity \(X=\sum_jH_jH^{-1}X\) and the trace-norm triangle
inequality give the result directly.  Semidefinite \(H_j\) follow by replacing
\(H_j\) by \(H_j+\varepsilon I\) and passing to the limit.
\end{proof}

\section{Proof of \texorpdfstring{Theorem~\ref{thm:main}}{Theorem 1.1}}\label{sec:proof-main}

We apply the partition inequality at a weighted maximizer to control angles
on the sphere, then use a planar comparison to reach the full cone.

\subsection{The weighted support inequality}\label{sec:support}

To evaluate the partition inequality at a maximizer, we need its Lagrange
condition and strict positivity.  We first compute the differential.

\begin{lemma}[Differential of the affinity]\label{lem:differential}
Let \(X\succ0\), \(D\succeq0\) and \(H=H^*\).  Then
\begin{equation}\label{eq:differential}
  D_X\Phi_p(X,D)[H]
  =t\,\trace\!\left[
  X^{-1/2}(X^{1/2}DX^{1/2})^tX^{-1/2}H
  \right].
\end{equation}
\end{lemma}

\begin{proof}
By symmetry, \(\Phi_p(X,D)=\trace(D^{1/2}XD^{1/2})^t\).
Let \(\Pi\) be the orthogonal projection onto \(\ran D\).  The matrix
\(D^{1/2}XD^{1/2}\) vanishes on \(\ran\Pi^{\perp}\).  Its compression to
\(\ran\Pi\) depends linearly on \(X\) and is positive definite whenever
\(X\succ0\).  Since \(M\mapsto M^t\) is real analytic on positive definite matrices,
\(\Phi_p(\cdot,D)\) is real analytic near \(X\), and differentiation under the
trace gives
\[
  D_X\Phi_p(X,D)[H]
  =t\,\trace\!\left[D^{1/2}(D^{1/2}XD^{1/2})^{t-1}D^{1/2}H\right],
\]
the power \((\cdot)^{t-1}\) being taken on the support.  Let
\(M=X^{1/2}D^{1/2}\), with singular value decomposition \(M=U\Sigma V^*\).
Then \(M(M^*M)^{t-1}M^*=U\Sigma^{2t}U^*=(MM^*)^t\), because
\(\sigma\cdot\sigma^{2t-2}\cdot\sigma=\sigma^{2t}\) for \(\sigma>0\) and both
sides vanish for \(\sigma=0\).  This identity shows that the coefficient of \(H\) equals
the one in \eqref{eq:differential}.
\end{proof}

\begin{lemma}[Positivity of a weighted maximizer]\label{lem:positive-maximizer}
Let \(a_1,\dots,a_m>0\) and \(B_1,\dots,B_m\succ0\).  A maximizer of
\(F(X):=\sum_{i=1}^ma_i\Phi_p(X,B_i)\) over \(X\succeq0\) subject to
\(\trace X^p=1\) is positive definite.
\end{lemma}

\begin{proof}
Existence follows from compactness.  Suppose that a maximizer \(X\) is
singular, let \(E\) be the orthogonal projection onto \(\ker X\), and put
\(k=\rank E>0\).  Let \(m_0>0\) be the smallest positive eigenvalue of \(X\).
For \(0<\varepsilon<m_0\), define \(X_\varepsilon:=X+\varepsilon E\).
Then \(X_\varepsilon\ge X\) and \(X_\varepsilon\ge\varepsilon I\).  For each
\(i\), set \(M_i:=B_i^{1/2}XB_i^{1/2}\) and
\(M_{i,\varepsilon}:=B_i^{1/2}X_\varepsilon B_i^{1/2}\).
Since \(B_i\succ0\), \(\rank M_i=\rank X=n-k\).  Eigenvalue monotonicity gives
\(\lambda_j(M_{i,\varepsilon})\ge\lambda_j(M_i)\) for all \(j\), while
\(M_{i,\varepsilon}\ge\varepsilon\lambda_{\min}(B_i)I\).  Hence
\begin{equation}\label{eq:kernel-gain}
  \trace M_{i,\varepsilon}^t
  \ge
  \trace M_i^t+k\bigl(\varepsilon\lambda_{\min}(B_i)\bigr)^t.
\end{equation}
Since \(\trace X_\varepsilon^p=1+k\varepsilon^p\), normalize by
\(\widetilde X_\varepsilon:=(1+k\varepsilon^p)^{-1/p}X_\varepsilon\).
The objective is homogeneous of degree \(t\) in its first argument.  By
\eqref{eq:kernel-gain},
\[
  F(\widetilde X_\varepsilon)
  \ge
  (1+k\varepsilon^p)^{-t/p}
  \left(F(X)+c\varepsilon^t\right),
  \qquad
  c:=k\sum_ia_i\lambda_{\min}(B_i)^t>0 .
\]
Since \(t=p/2<p\), the positive \(\varepsilon^t\) term dominates the
\(O(\varepsilon^p)\) normalization loss.  Thus
\(F(\widetilde X_\varepsilon)>F(X)\) for all sufficiently small
\(\varepsilon\), a contradiction.
\end{proof}

We apply these lemmas to a weighted sum of two affinities.

\begin{prop}[Weighted support inequality]\label{prop:support}
Let \(A,B,C\in\sphere\).  Then, for every \(a,b\ge0\),
\begin{equation}\label{eq:support}
  a\Phi_p(A,B)+b\Phi_p(A,C)
  \le
  \sqrt{a^2+b^2+2ab\Phi_p(B,C)}.
\end{equation}
\end{prop}

\begin{proof}
We first assume \(a,b>0\) and \(B,C\succ0\).  Maximize
\(F(X)=a\Phi_p(X,B)+b\Phi_p(X,C)\) over \(\sphere\), and let \(A_*\) be a
maximizer.  By Lemma~\ref{lem:positive-maximizer}, \(A_*\succ0\).  Put
\(Y:=A_*^{1/2}BA_*^{1/2}\) and \(Z:=A_*^{1/2}CA_*^{1/2}\).
By Lemma~\ref{lem:differential}, the Lagrange multiplier equation is
\begin{equation}\label{eq:weighted-balance}
  aY^t+bZ^t=\lambda A_*^p.
\end{equation}
Taking traces and using \(\trace A_*^p=1\) identifies
\(\lambda=a\Phi_p(A_*,B)+b\Phi_p(A_*,C)=F(A_*)\).

Set \(P:=aY^t\), \(Q:=bZ^t\), and \(T:=P+Q=\lambda A_*^p\), and abbreviate
\(x_*:=\Phi_p(A_*,B)\), \(y_*:=\Phi_p(A_*,C)\), and \(z:=\Phi_p(B,C)\).
Since \(t/p=1/2\), we have \(P^{1/p}=a^{1/p}Y^{1/2}\),
\(Q^{1/p}=b^{1/p}Z^{1/2}\), and \(T^{-1/p}=\lambda^{-1/p}A_*^{-1}\).
These identities give the exact norm evaluations
\begin{align}
  \pnorm{Q^{1/p}T^{-1/p}P^{1/p}}{p}^p&=\frac{ab}{\lambda}\,z,
  \label{eq:cross-exact}\\
  \pnorm{P^{1/p}T^{-1/p}P^{1/p}}{p}^p&=\frac{a^2}{\lambda},
  \label{eq:self-P-exact}\\
  \pnorm{Q^{1/p}T^{-1/p}Q^{1/p}}{p}^p&=\frac{b^2}{\lambda}.
  \label{eq:self-Q-exact}
\end{align}
For the cross identity, put \(X:=Z^{1/2}A_*^{-1}Y^{1/2}\).  Then \(X^*X\) has
the same nonzero spectrum as
\(YA_*^{-1}ZA_*^{-1}=A_*^{1/2}BCA_*^{-1/2}\), which is similar to \(BC\).
Therefore \(\pnorm{X}{p}^p=\trace(X^*X)^{p/2}=\Phi_p(B,C)\), and
\eqref{eq:cross-exact} follows from the scalar factors above.  For the first
self identity, the positive matrices \(Y^{1/2}A_*^{-1}Y^{1/2}\) and
\(A_*^{-1/2}YA_*^{-1/2}=B\) have the same nonzero spectrum, so their
\(p\)th power traces agree.  The
calculation for \(C\) is the same.  The
normalization \(\trace B^p=\trace C^p=1\) then yields
\eqref{eq:self-P-exact}--\eqref{eq:self-Q-exact}.

Apply Theorem~\ref{thm:partition} to \(T=Q+P\) with \(X=P^{1/p}\).  Using
\eqref{eq:cross-exact} and \eqref{eq:self-P-exact},
\[
  \frac{abz+a^2}{\lambda}
  \ge\pnorm{P^{1/p}}{p}^p
  =\trace P=ax_*.
\]
Thus \(a^2+abz\ge\lambda a x_*\).  Applying the same inequality with
\(X=Q^{1/p}\) gives \(b^2+abz\ge\lambda b y_*\).  Adding and using
\(\lambda=ax_*+by_*\),
\begin{equation}\label{eq:lambda-square}
  a^2+b^2+2abz
  \ge
  \lambda(ax_*+by_*)=\lambda^2.
\end{equation}
Since \(\lambda=\max_{X\in\sphere}F(X)\), evaluating at the given \(A\)
proves \eqref{eq:support}.
If \(B\) or \(C\) is singular, replace it by a positive definite normalized
approximation and pass to the limit.  If \(a=0\) or \(b=0\), the claim follows
directly from \eqref{eq:holder-affinity}.
\end{proof}

\subsection{The angular triangle inequality}\label{sec:angular}

To turn the support estimate into a triangle inequality, we express the
affinity as the cosine of an angle: on \(\sphere\), define
\(\vartheta_p(A,B):=\arccos\Phi_p(A,B)\).
By \eqref{eq:holder-affinity}, \(\vartheta_p(A,B)\in[0,\pi/2]\).

\begin{theorem}[Angular triangle inequality]\label{thm:angular-triangle}
For all \(A,B,C\in\sphere\),
\begin{equation}\label{eq:angular-triangle}
  \vartheta_p(B,C)
  \le
  \vartheta_p(B,A)+\vartheta_p(A,C).
\end{equation}
\end{theorem}

\begin{proof}
Set \(\alpha:=\vartheta_p(A,B)\), \(\beta:=\vartheta_p(A,C)\), and
\(\gamma:=\vartheta_p(B,C)\).  Thus \(\Phi_p(A,B)=\cos\alpha\),
\(\Phi_p(A,C)=\cos\beta\), and \(\Phi_p(B,C)=\cos\gamma\).
If \(\alpha+\beta\ge\pi/2\), then \eqref{eq:angular-triangle} is automatic
because \(\gamma\le\pi/2\).  Assume therefore that \(\alpha+\beta<\pi/2\).

First suppose \(\alpha,\beta>0\).  Apply Proposition~\ref{prop:support} with
\(a=\sin\beta\) and \(b=\sin\alpha\).  The left-hand side becomes
\(\sin\beta\cos\alpha+\sin\alpha\cos\beta=\sin(\alpha+\beta)\).
If \(\gamma>\alpha+\beta\), then the difference between the square of the
right-hand side of \eqref{eq:support} and \(\sin^2(\alpha+\beta)\) is
\begin{equation}\label{eq:angle-difference}
  2\sin\alpha\sin\beta
  \bigl(\cos\gamma-\cos(\alpha+\beta)\bigr)<0,
\end{equation}
contradicting \eqref{eq:support}.  Hence \(\gamma\le\alpha+\beta\).

If \(\alpha=0\), apply \eqref{eq:support} with \(a>0\) and \(b>0\), divide the
squared inequality by \(ab\), and let \(b/a\downarrow0\).  This gives
\(\cos\beta\le\cos\gamma\), so \(\gamma\le\beta\).  The case \(\beta=0\)
is symmetric.
\end{proof}

\begin{corr}[Chordal triangle inequality on the normalized sphere]
\label{cor:normalized-triangle}
For \(A,B,C\in\sphere\),
\begin{equation}\label{eq:normalized-triangle}
  D_p(A,C)\le D_p(A,B)+D_p(B,C).
\end{equation}
\end{corr}

\begin{proof}
On \(\sphere\),
\begin{equation}\label{eq:chord-angle}
  D_p(A,B)=\sqrt{2-2\cos\vartheta_p(A,B)}
  =2\sin\frac{\vartheta_p(A,B)}2.
\end{equation}
Apply Theorem~\ref{thm:angular-triangle} and the elementary inequality
\(\sin\frac{u+v}{2}\le\sin\frac u2+\sin\frac v2\) for \(u,v\ge0\) with
\(u+v\le\pi\).
\end{proof}

\subsection{From the sphere to the cone}\label{sec:conic}

\begin{proof}[\textbf{Proof of Theorem~\ref{thm:main}}]
First assume \(A,B,C\ne0\).  Let \(r_A:=r_p(A)\), \(r_B:=r_p(B)\), and
\(r_C:=r_p(C)\).  Normalization as in \eqref{eq:normalization} gives the law of cosines
\begin{equation}\label{eq:cone-law-cosines}
  D_p(A,B)^2
  =r_A^2+r_B^2
  -2r_Ar_B\cos\vartheta_p(\what A,\what B).
\end{equation}
Put \(a:=\vartheta_p(\what A,\what B)\),
\(b:=\vartheta_p(\what B,\what C)\), and
\(c:=\vartheta_p(\what A,\what C)\).
By Theorem~\ref{thm:angular-triangle}, \(c\le a+b\).  Since all three angles lie
in \([0,\pi/2]\), we have \(a+b\le\pi\).

Choose planar vectors \(u_A:=r_A(1,0)\), \(u_B:=r_B(\cos a,\sin a)\), and
\(u_C:=r_C(\cos(a+b),\sin(a+b))\).  By \eqref{eq:cone-law-cosines},
\(\norm{u_A-u_B}=D_p(A,B)\) and \(\norm{u_B-u_C}=D_p(B,C)\).
Moreover, the Euclidean law-of-cosines expression is increasing in its angle on
\([0,\pi]\), so \(c\le a+b\) implies \(D_p(A,C)\le\norm{u_A-u_C}\).
The ordinary Euclidean triangle inequality therefore gives
\begin{align*}
  D_p(A,C)
  &\le\norm{u_A-u_C}\\
  &\le\norm{u_A-u_B}+\norm{u_B-u_C}\\
  &=D_p(A,B)+D_p(B,C).
\end{align*}
The cases involving the zero matrix follow by continuity, or directly from
\(D_p(A,0)=r_p(A)\).

Symmetry and nonnegativity were noted in \eqref{eq:affinity-definition} and
\eqref{eq:holder-affinity}.  Finally, if \(D_p(A,B)=0\), then
\eqref{eq:cone-law-cosines} forces \(r_A=r_B\) and equality in Schatten
H\"older in \eqref{eq:holder-affinity}.  The standard equality characterization
gives \(A^p\) proportional to \(B^p\); the equal radii make the proportionality
constant one, hence \(A=B\).
\end{proof}

\section{The LvD conjecture}\label{sec:lin-block-matrix}

To prove Theorem~\ref{thm:lin-block-norm}, we express the block-norm entries
as affinities and use the angular triangle inequality.  The conjecture of
\citet[Conjecture~5.1]{lin2014positive} was known at both endpoints:
\(p=1\) is a theorem of \citet{drury2014positive}, with a short proof in
\citet{li2015positivity}, and \(p=2\) follows from
\citet[Theorem~1]{marcus1971partitioned}, for any number of blocks.  For
scalar blocks, it follows from the fractional Hadamard power theorem of
\citet{fitzgerald1977fractional}.  The \(p=1\) statement had also been asserted
for matrices of fidelities by \citet{fannes2012matrices}, with an argument
that \citet{lin2014positive} found incomplete.

The \(p\)th powers in the conjecture are essential: \citet{li2015positivity}
give a positive semidefinite block matrix for which
\([\pnorm{H_{ij}}{p}]\) is indefinite for every \(p>1\).
For the powered entries, the following three-point criterion converts the
angular triangle inequality into positivity of a cosine matrix.

\begin{lemma}[Three-point cosine criterion]\label{lem:three-point-cosine-criterion}
Let \(c_{12},c_{13},c_{23}\in[0,1]\), and set
\[
  C:=
  \begin{bmatrix}
    1&c_{12}&c_{13}\\
    c_{12}&1&c_{23}\\
    c_{13}&c_{23}&1
  \end{bmatrix}.
\]
Define \(\alpha:=\arccos c_{12}\), \(\beta:=\arccos c_{13}\), and
\(\gamma:=\arccos c_{23}\).
Then \(C\succeq0\) if and only if
\begin{equation}\label{eq:three-point-angle-criterion}
  |\alpha-\beta|\le\gamma\le\alpha+\beta.
\end{equation}
\end{lemma}

\begin{proof}
The \(1\times1\) and \(2\times2\) principal minors of \(C\) are nonnegative
because \(0\le c_{ij}\le1\).  Since a Hermitian matrix is positive semidefinite
exactly when all its principal minors are nonnegative, \(C\succeq0\) is
equivalent to \(\det C\ge0\).  Directly,
\begin{equation}\label{eq:lin-correlation-determinant}
  \det C
  =1+2c_{12}c_{13}c_{23}-c_{12}^2-c_{13}^2-c_{23}^2.
\end{equation}
Regarded as a quadratic polynomial in \(c_{23}\) with leading coefficient
\(-1\), its roots are
\(c_{12}c_{13}\pm\sqrt{(1-c_{12}^2)(1-c_{13}^2)}=\cos(\alpha\mp\beta)\).
Therefore
\begin{equation}\label{eq:lin-determinant-factorization}
  \det C
  =
  \bigl(\cos|\alpha-\beta|-\cos\gamma\bigr)
  \bigl(\cos\gamma-\cos(\alpha+\beta)\bigr).
\end{equation}
The determinant is nonnegative exactly when \(\cos\gamma\) lies between
the two roots.  Since \(0\le|\alpha-\beta|\le\alpha+\beta\le\pi\) and cosine
is decreasing on \([0,\pi]\), this is equivalent to
\eqref{eq:three-point-angle-criterion}.
\end{proof}

\begin{proof}[\textbf{Proof of Theorem~\ref{thm:lin-block-norm}}]
Choose a block Gram factorization \(\mathbf H=X^*X\), with
\(X=[\,X_1\ \ X_2\ \ X_3\,]\), for instance by taking
\(X=\mathbf H^{1/2}\) and splitting its columns conformally with the block
structure.  Then \(H_{ij}=X_i^*X_j\).
All matrices \(X_i\) have a common row dimension, although their column
dimensions may differ.  Define positive semidefinite matrices on this common
row space by \(A_i:=X_iX_i^*\) for \(i=1,2,3\).  The block-norm entries
are their affinities:
\begin{align}
  \pnorm{H_{ij}}{p}^p
  &=\trace\left|X_i^*X_j\right|^p\notag\\
  &=\trace\left(X_j^*X_iX_i^*X_j\right)^{p/2}\notag\\
  &=\trace\left(A_i^{1/2}A_jA_i^{1/2}\right)^{p/2}\notag\\
  &=\Phi_p(A_i,A_j).
  \label{eq:lin-entry-affinity-proof}
\end{align}
The third equality holds because, for \(M=A_i^{1/2}X_j\), the positive
matrices \(M^*M=X_j^*A_iX_j\) and
\(MM^*=A_i^{1/2}X_jX_j^*A_i^{1/2}=A_i^{1/2}A_jA_i^{1/2}\) have the same
nonzero eigenvalues, counted with multiplicity.

Set \(r_i:=r_p(A_i)=(\trace A_i^p)^{1/2}\).
If \(r_i=0\), then \(A_i=0\), hence \(X_i=0\), and the \(i\)th row and column
of the matrix in \eqref{eq:lin-block-norm-psd} vanish.  We may therefore remove
such indices and assume \(r_i>0\) for all \(i\).  Normalize as in
\eqref{eq:normalization}, setting \(\what A_i:=r_i^{-2/p}A_i\in\sphere\).
Radial homogeneity gives
\(\Phi_p(A_i,A_j)=r_ir_j\Phi_p(\what A_i,\what A_j)\).
Consequently,
\begin{equation}\label{eq:lin-diagonal-congruence}
  \bigl[\pnorm{H_{ij}}{p}^p\bigr]_{i,j=1}^3
  =DCD,
  \qquad
  D:=\diag(r_1,r_2,r_3),
\end{equation}
where \(C:=[\Phi_p(\what A_i,\what A_j)]_{i,j=1}^3\).
The diagonal entries are \(\Phi_p(\what A_i,\what A_i)=\trace\what A_i^p=1\),
and its off-diagonal entries \(c_{ij}:=\Phi_p(\what A_i,\what A_j)\) lie in
\([0,1]\) by \eqref{eq:holder-affinity}.

Define \(\theta_{ij}:=\vartheta_p(\what A_i,\what A_j)=\arccos c_{ij}\).
Applying Theorem~\ref{thm:angular-triangle} to all three orderings of the triple
gives \(|\theta_{12}-\theta_{13}|\le\theta_{23}\le\theta_{12}+\theta_{13}\).
Therefore Lemma~\ref{lem:three-point-cosine-criterion} yields \(C\succeq0\).
Finally, \eqref{eq:lin-diagonal-congruence} gives
\([\pnorm{H_{ij}}{p}^p]_{i,j=1}^3=DCD\succeq0\), as required.
\end{proof}

\begin{rmk}[Semidefinite versus definite conclusions]\label{rem:lin-semidefinite}
If \(\mathbf H\succ0\), then \([\pnorm{H_{ij}}{p}^p]\succ0\).
Indeed, choose \(\varepsilon>0\) such that \(\mathbf H-\varepsilon I\succeq0\)
and apply Theorem~\ref{thm:lin-block-norm} to this matrix.  Its off-diagonal
blocks are unchanged, while each diagonal entry of the scalar matrix is
strictly smaller.  Thus \([\pnorm{H_{ij}}{p}^p]\) is a positive semidefinite
matrix plus a strictly positive diagonal matrix.
\end{rmk}

The following singular configuration also attains equality in the angular
triangle inequality.

\begin{example}[A tight configuration, for every \(p\)]\label{ex:lin-tight}
Take \(n=2\), let \(u=(1,1)^{\mathsf T}/\sqrt2\) and
\(v=(1,-1)^{\mathsf T}/\sqrt2\), and set \(A_1:=uu^*\), \(A_2:=vv^*\),
and \(A_3:=2^{-1/p}I_2\).
All three lie on \(\sphere\), since \(\trace A_1^p=\trace A_2^p=1\) and
\(\trace A_3^p=2\cdot2^{-1}=1\).  For a rank-one \(A=ww^*\) with
\(\norm w=1\) and any \(B\succeq0\) one has \(A^{1/2}=A\) and
\(ABA=(w^*Bw)ww^*\), so \(\Phi_p(A,B)=(w^*Bw)^{p/2}\).  Hence
\[
  \Phi_p(A_1,A_2)=|\ip{u}{v}|^p=0,
  \qquad
  \Phi_p(A_1,A_3)=\Phi_p(A_2,A_3)=\bigl(2^{-1/p}\bigr)^{p/2}
  =\tfrac1{\sqrt2},
\]
that is, \(\vartheta_p(A_1,A_2)=\pi/2\) and
\(\vartheta_p(A_1,A_3)=\vartheta_p(A_2,A_3)=\pi/4\).  The angular triangle
inequality \eqref{eq:angular-triangle} therefore holds with \emph{equality},
for every \(1\le p\le2\); correspondingly the cosine matrix \(C\) has
\(\det C=0\), so the matrix
\eqref{eq:lin-block-norm-psd} built from \(H_{ij}=A_i^{1/2}A_j^{1/2}\) is
positive semidefinite and singular.
\end{example}

\begin{rmk}[The range of \(p\)]\label{rem:lin-sharp-in-p}
For \(n\ge2\), the triangle inequality for \(D_p\) fails for every \(p>2\)
\citep[Theorem~1.3(ii)]{vuong2026three}.  The block-positivity statement also
fails in this range: applied to \(H_{ij}=A_i^{1/2}A_j^{1/2}\), it would make
\([\Phi_p(A_i,A_j)]_{i,j=1}^3\) a Gram matrix, so the distances \(D_p(A_i,A_j)\)
would be Euclidean distances and would satisfy the triangle inequality.
\end{rmk}

\section{Barycenters}\label{sec:barycenters}

To study Fr\'echet means for \(D_p\), fix \(A_1,\dots,A_m\in\pp_n\), not all
zero, and weights in \(\simplex:=\{w\in\reals^m:w_i\ge0,\ \sum_iw_i=1\}\).
The \emph{Fr\'echet functional} is
\(\Fc_w(X):=\sum_{i=1}^mw_iD_p(X,A_i)^2\) for \(X\in\pp_n\); any minimizer
is a \emph{barycenter} of \((A_i,w_i)\) for \(D_p\).
Expanding \eqref{eq:distance-definition},
\begin{equation}\label{eq:frechet-expansion}
  \Fc_w(X)=\trace X^p-2\Psi_w(X)+\sum_{i=1}^mw_i\trace A_i^p,
  \qquad
  \Psi_w(X):=\sum_{i=1}^mw_i\Phi_p(X,A_i).
\end{equation}
The affinity sum \(\Psi_w\) is positively homogeneous of degree \(t=p/2\),
while \(\trace X^p\) is homogeneous of degree \(p\).
For \(m=2\) and \(w=(a,b)\), \(\Psi_w\) is the objective in
Proposition~\ref{prop:support}.

\subsection{Convexity, existence, uniqueness}\label{sec:convexity}

\begin{prop}[Convexity]\label{prop:convexity}
Let \(1\le p\le2\).
\begin{enumerate}[label=\textup{(\roman*)},leftmargin=2.2em]
\item For every \(A\in\pp_n\), the map \(X\mapsto\Phi_p(X,A)\) is concave and
nondecreasing on \(\pp_n\).
\item The map \(X\mapsto\trace X^p\) is convex on \(\pp_n\), and strictly convex when
\(p>1\).
\item For every \(w\in\simplex\), \(\Fc_w\) is convex on \(\pp_n\), and
strictly convex when \(p>1\).
\end{enumerate}
\end{prop}

\begin{proof}
(i)  By \eqref{eq:affinity-definition} and the coincidence of nonzero spectra,
\(\Phi_p(X,A)=\trace(A^{1/2}XA^{1/2})^t\) with \(t=p/2\in[\half,1]\).  The map
\(X\mapsto A^{1/2}XA^{1/2}\) is linear and maps \(\pp_n\) into \(\pp_n\), and
\(M\mapsto\trace M^t\) is concave and nondecreasing on \(\pp_n\) for
\(t\in[0,1]\), since \(\tau\mapsto\tau^t\) is concave and nondecreasing on
\([0,\infty)\) \citep{bhatia1997matrix,carlen2010trace}.
Composing these maps proves (i).

(ii)  On \(\pp_n\) we have \(\trace X^p=\pnorm{X}{p}^p\), and \(\pnorm\cdot p\)
is a norm for \(p\ge1\), so convexity follows from convexity and monotonicity
of \(\tau\mapsto\tau^p\) on \([0,\infty)\).  For strictness let \(p>1\), let
\(X\ne Y\) in \(\pp_n\), let \(\lambda\in(0,1)\), and set
\(Z=\lambda X+(1-\lambda)Y\).  If \(\pnorm Xp\ne\pnorm Yp\), then
\[
  \pnorm Zp^p
  \le\bigl(\lambda\pnorm Xp+(1-\lambda)\pnorm Yp\bigr)^p
  <\lambda\pnorm Xp^p+(1-\lambda)\pnorm Yp^p
\]
by strict convexity of \(\tau\mapsto\tau^p\).  If
\(\pnorm Xp=\pnorm Yp=:\rho\), then \(\rho>0\) because \(X\ne Y\), and the
Schatten class \(S_p\) is uniformly, hence strictly, convex for
\(1<p<\infty\) \citep{mccarthy1967cp,simon2005trace}; therefore
\(\pnorm Zp<\rho\) and \(\pnorm Zp^p<\rho^p\).

(iii)  By \eqref{eq:frechet-expansion}, \(\Fc_w\) is the sum of the convex
functions \(\trace X^p\) and \(-2\Psi_w\), and a constant.
Strict convexity for \(p>1\) comes from the first summand.
\end{proof}

\begin{theorem}[Existence and uniqueness]\label{thm:existence}
Let \(1\le p\le2\) and \(w\in\simplex\).  Then \(\Fc_w\) has a nonempty
compact convex set of minimizers on \(\pp_n\), and the minimizer is unique
if \(1<p\le2\).
\end{theorem}

\begin{proof}
By \eqref{eq:holder-affinity}, \(D_p(X,A_i)^2\ge(r_p(X)-r_p(A_i))^2\), so
\(\Fc_w(X)\to\infty\) as \(r_p(X)\to\infty\).  By continuity on the closed
cone \(\pp_n\), its sublevel sets are compact, so it attains its minimum.
The minimizer set is compact; by Proposition~\ref{prop:convexity}, it is
convex and reduces to a point for \(p>1\).
\end{proof}

At \(p=1\) the term \(\trace X\) is linear, so
Proposition~\ref{prop:convexity}(iii) gives convexity but not strict convexity.
Nevertheless, the barycenter is unique whenever some \(A_i\) with \(w_i>0\)
is nonsingular \citep{agueh2011barycenters,alvarez2016fixed}.
For the Euler characterization, we next show that positive definite data
give interior minimizers.

\begin{lemma}[Positivity of the barycenter]\label{lem:positive-barycenter}
Let \(1\le p\le2\), \(w\in\simplex\) with all \(w_i>0\), and suppose
\(A_i\succ0\) for every \(i\).  Then every barycenter is positive definite.
\end{lemma}

\begin{proof}
Suppose \(X\in\pp_n\) is singular, let \(E\) be the orthogonal projection onto
\(\ker X\), \(k=\rank E>0\), and let \(m_0>0\) be the smallest positive
eigenvalue of \(X\) (take \(m_0=1\) if \(X=0\)).  For \(0<\varepsilon<m_0\) put
\(X_\varepsilon:=X+\varepsilon E\), so that
\(X_\varepsilon\succeq\varepsilon I\).  Since \(E\) is supported on
\(\ker X\), we have \(\trace X_\varepsilon^p=\trace X^p+k\varepsilon^p\).
As in \eqref{eq:kernel-gain}, each affinity satisfies
\(\Phi_p(X_\varepsilon,A_i)\ge\Phi_p(X,A_i)
+k(\varepsilon\lambda_{\min}(A_i))^t\).
Hence, by \eqref{eq:frechet-expansion},
\[
  \Fc_w(X_\varepsilon)-\Fc_w(X)
  \le k\varepsilon^p-2k\varepsilon^{t}\sum_iw_i\lambda_{\min}(A_i)^t .
\]
Since \(t=p/2<p\), this is negative for sufficiently small \(\varepsilon>0\),
so \(X\) cannot minimize \(\Fc_w\).
\end{proof}

\subsection{The Euler equation}\label{sec:euler}

\begin{proof}[\textbf{Proof of Theorem~\ref{thm:euler}}]
By Lemma~\ref{lem:differential}, each \(\Phi_p(\cdot,A_i)\) is real analytic
near \(X\succ0\), while \(X\mapsto\trace X^p\) has differential
\(H\mapsto p\,\trace[X^{p-1}H]\).  Thus \eqref{eq:frechet-expansion}, with
\(2t=p\), gives
\begin{equation}\label{eq:frechet-gradient}
  \nabla\Fc_w(X)
  =p\left(X^{p-1}-X^{-1/2}
  \Bigl(\sum_iw_i(X^{1/2}A_iX^{1/2})^{p/2}\Bigr)X^{-1/2}\right).
\end{equation}
Conjugating by \(X^{1/2}\) shows that \(\nabla\Fc_w(X)=0\) is equivalent to
\eqref{eq:euler}.

A positive definite \(X\) is interior to \(\pp_n\) in the real vector space
of Hermitian matrices, so any such minimizer has vanishing gradient.
Conversely, convexity (Proposition~\ref{prop:convexity}(iii)) makes every
stationary point a global minimizer.
\end{proof}

\subsection{Convergence of the fixed-point iteration}\label{sec:convergence}

Equation \eqref{eq:euler} suggests the iteration
\begin{equation}\label{eq:fixed-point-map}
  \Tc(X):=X^{-1/2}
  \Bigl(\sum_{i=1}^mw_i\bigl(X^{1/2}A_iX^{1/2}\bigr)^{p/2}\Bigr)^{2/p}
  X^{-1/2},
\end{equation}
whose fixed points in the positive definite cone are exactly the barycenters
there, by Theorem~\ref{thm:euler}.  At
\(p=2\) the map is constant, \(\Tc(X)=\sum_iw_iA_i\), so it converges in one
step.  For real symmetric data at \(p=1\), it is the Gaussian Wasserstein
barycenter iteration of \citet{alvarez2016fixed}; linear convergence rates
were established by \citet{chewi2020gradient}.

For the full range \(1\le p\le2\), we combine sufficient decrease from the
partition inequality (Lemma~\ref{lem:sufficient-decrease}),
confinement of the iterates to a spectral box
(Lemma~\ref{lem:spectral-confinement}), and an error bound that converts the
step length into control of the optimality gap
(Lemma~\ref{lem:residual-error-bound}).  Appendix~\ref{app:convergence}
supplies the confinement and error-bound proofs, including \(p=1\);
we prove sufficient decrease below.

Throughout the rest of this subsection, \(A_1,\dots,A_m\succ0\) and
\(w\in\simplex\) has all \(w_i>0\).  We abbreviate the right side of the
Euler equation \eqref{eq:euler} by
\begin{equation}\label{eq:balance-map}
  S(X):=\sum_{i=1}^mw_i\bigl(X^{1/2}A_iX^{1/2}\bigr)^{p/2},
  \qquad\text{so that}\qquad
  \Tc(X)=X^{-1/2}S(X)^{1/t}X^{-1/2}.
\end{equation}
Since \(S(X)\succ0\) whenever \(X\succ0\), the iteration is well defined:
\begin{equation}\label{eq:iteration}
  X_{k+1}:=\Tc(X_k),
  \qquad
  X_0\succ0.
\end{equation}

\begin{lemma}[Sufficient decrease]\label{lem:sufficient-decrease}
For every \(X\succ0\),
\begin{equation}\label{eq:sufficient-decrease}
  \boxed{
  \Fc_w\bigl(\Tc(X)\bigr)+D_p\bigl(X,\Tc(X)\bigr)^2\ \le\ \Fc_w(X).
  }
\end{equation}
In particular \(\Fc_w\) is nonincreasing along \eqref{eq:iteration}, strictly so
unless \(\Tc(X)=X\).
\end{lemma}

\begin{proof}
Put \(K_i:=(X^{1/2}A_iX^{1/2})^{1/2}\) and
\(K:=(\sum_{i=1}^mw_iK_i^p)^{1/p}\), so that \(K^p=S(X)\) and
\(\Tc(X)=X^{-1/2}K^2X^{-1/2}\).  Since
\(X^{1/2}\Tc(X)X^{1/2}=K^2\), we have the exact identity
\begin{equation}\label{eq:phi-X-T}
  \Phi_p\bigl(X,\Tc(X)\bigr)
  =\trace K^p
  =\sum_{i=1}^mw_i\Phi_p(X,A_i).
\end{equation}
This cancels the affinity terms in the decrease estimate.  For the remaining
terms, use
\begin{equation}\label{eq:phi-T-Ai}
  \Phi_p\bigl(\Tc(X),A_i\bigr)=\pnorm{K_iX^{-1}K}{p}^p,
  \qquad
  \trace\Tc(X)^p=\pnorm{KX^{-1}K}{p}^p.
\end{equation}
For the first, \(A_i=X^{-1/2}K_i^2X^{-1/2}\), so \(\Tc(X)A_i\) is similar to
\(K^2X^{-1}K_i^2X^{-1}\), which has the same nonzero spectrum as
\(KX^{-1}K_i^2X^{-1}K=(K_iX^{-1}K)^*(K_iX^{-1}K)\);
taking the \((p/2)\)th power trace gives the claim.  For the second,
\(\Tc(X)=(X^{-1/2}K)(KX^{-1/2})\) and \(KX^{-1}K=(KX^{-1/2})(X^{-1/2}K)\) have
the same nonzero spectrum, and both are positive semidefinite.

Apply Theorem~\ref{thm:partition} with \(H_i:=w_iK_i^p\),
\(H:=\sum_iH_i=K^p\), and \(KX^{-1}K\) as the free argument.  Here
\(H_i^{1/p}=w_i^{1/p}K_i\) and \(H^{-1/p}=K^{-1}\), so \eqref{eq:partition-intro}
reads
\begin{equation}\label{eq:key-partition-bary}
  \sum_{i=1}^mw_i\pnorm{K_iX^{-1}K}{p}^p\ \ge\ \pnorm{KX^{-1}K}{p}^p,
\end{equation}
which by \eqref{eq:phi-T-Ai} is exactly
\begin{equation}\label{eq:key-affinity}
  \sum_{i=1}^mw_i\Phi_p\bigl(\Tc(X),A_i\bigr)\ \ge\ \trace\Tc(X)^p.
\end{equation}
Expanding by \eqref{eq:distance-definition} and cancelling with
\eqref{eq:phi-X-T} gives
\begin{equation}\label{eq:exact-descent-remainder}
  \Fc_w(X)-\Fc_w\bigl(\Tc(X)\bigr)-D_p\bigl(X,\Tc(X)\bigr)^2
  =2\left[\sum_{i=1}^mw_i\Phi_p\bigl(\Tc(X),A_i\bigr)-\trace\Tc(X)^p\right],
\end{equation}
whose right side is nonnegative by \eqref{eq:key-affinity}.  If
\(\Tc(X)\ne X\) then \(D_p(X,\Tc(X))>0\), because \(D_p\) is a metric
(Theorem~\ref{thm:main}), so the decrease is strict.
\end{proof}

\begin{rmk}[What the proof uses]\label{rem:descent-hypotheses}
The proof also applies to semidefinite \(A_i\) provided \(X\succ0\) and
\(S(X)\succ0\).
By \eqref{eq:exact-descent-remainder}, the remainder in the decrease estimate
is twice the surplus in the partition inequality
\eqref{eq:key-partition-bary}.  At \(p=2\) both inequalities hold with equality.
\end{rmk}

Lemma~\ref{lem:sufficient-decrease} gives convergence of the objective
values.  To obtain a rate, we first bound the spectra of the iterates
away from zero and infinity.

\begin{lemma}[Uniform spectral confinement]\label{lem:spectral-confinement}
Fix \(X_0\succ0\), let \(F_0:=\Fc_w(X_0)\), and put
\begin{align}
  \bar r&:=\sum_{i=1}^mw_ir_p(A_i), &
  \lambda_+&:=\bigl(\bar r+\sqrt{F_0}\bigr)^{2/p},
  \label{eq:box-upper}\\
  c_A&:=\sum_{i=1}^mw_i\lambda_{\min}(A_i)^t, &
  \lambda_-&:=\frac{c_A^{\,n/t}}{\lambda_+^{\,n-1}}.
  \label{eq:box-lower}
\end{align}
Then every iterate of \eqref{eq:iteration} with \(k\ge1\), and every barycenter
\(X_\star\), lies in the spectral box
\(\Bc:=\{X\in\pp_n:\lambda_-I\preceq X\preceq\lambda_+I\}\).
\end{lemma}
Within this box, the step length \(D_p(X,\Tc(X))\) controls the optimality
gap.  This step uses \(p>1\); Proposition~\ref{prop:bures-endpoint-rate}
treats \(p=1\).

\begin{lemma}[Residual error bound]\label{lem:residual-error-bound}
Let \(1<p\le2\) and let \(X\in\Bc\) with \(\Tc(X)\in\Bc\).  Then
\begin{equation}\label{eq:error-bound}
  D_p\bigl(X,\Tc(X)\bigr)^2
  \ \ge\
  \rho_p\bigl(\Fc_w(X)-\Fc_w(X_\star)\bigr),
  \qquad
  \rho_p:=\frac{4(p-1)^2}{p^2}\left(\frac{\lambda_-}{\lambda_+}\right)^{6-2p},
\end{equation}
where \(X_\star\) is the barycenter.  Moreover \(0<\rho_p\le1\).
\end{lemma}
Appendix~\ref{app:convergence} proves both lemmas; combined with
Lemma~\ref{lem:sufficient-decrease}, they give linear convergence.

\begin{theorem}[Global linear convergence]\label{thm:linear-convergence}
Let \(1<p\le2\), let \(A_1,\dots,A_m\succ0\), let \(w\in\simplex\) have
positive coordinates, let \(X_0\succ0\), and let \(X_\star\) be the barycenter.
With \(\lambda_\pm\) as in \eqref{eq:box-upper}--\eqref{eq:box-lower} and
\(\rho_p\) as in \eqref{eq:error-bound}, the iteration \eqref{eq:iteration}
satisfies, for every \(k\ge1\),
\begin{equation}\label{eq:linear-rate}
  \boxed{
  \Fc_w(X_k)-\Fc_w(X_\star)
  \ \le\
  (1-\rho_p)^{k-1}\bigl(\Fc_w(X_1)-\Fc_w(X_\star)\bigr),
  }
\end{equation}
so the objective values converge \(Q\)-linearly, and the iterates converge
\(R\)-linearly,
\begin{equation}\label{eq:linear-iterates}
  \frob{X_k-X_\star}
  \ \le\
  \left(\frac{2\bigl(\Fc_w(X_1)-\Fc_w(X_\star)\bigr)}
             {p(p-1)\lambda_+^{p-2}}\right)^{1/2}
  (1-\rho_p)^{(k-1)/2}.
\end{equation}
Moreover the computable step length certifies the gap a posteriori:
\begin{equation}\label{eq:posteriori}
  \Fc_w(X_k)-\Fc_w(X_\star)\ \le\ \rho_p^{-1}D_p(X_k,X_{k+1})^2 .
\end{equation}
\end{theorem}

\begin{proof}
By Lemma~\ref{lem:spectral-confinement}, \(X_k\in\Bc\) for every \(k\ge1\), and
\(X_\star\in\Bc\).  For \(k\ge1\), Lemma~\ref{lem:sufficient-decrease} followed
by Lemma~\ref{lem:residual-error-bound} at \(X=X_k\) gives
\[
  \Fc_w(X_{k+1})-\Fc_w(X_\star)
  \le\Fc_w(X_k)-\Fc_w(X_\star)-D_p(X_k,X_{k+1})^2
  \le(1-\rho_p)\bigl(\Fc_w(X_k)-\Fc_w(X_\star)\bigr),
\]
and iterating from \(k=1\) proves \eqref{eq:linear-rate}.  Strong convexity of
\(\Fc_w\) on \(\Bc\) with modulus \(\mu_F=p(p-1)\lambda_+^{p-2}\), together
with \(\nabla\Fc_w(X_\star)=0\), gives
\(\Fc_w(X)-\Fc_w(X_\star)\ge\tfrac12\mu_F\frob{X-X_\star}^2\)
on \(\Bc\), which with \eqref{eq:linear-rate} gives
\eqref{eq:linear-iterates}.  Estimate \eqref{eq:posteriori} is
\eqref{eq:error-bound} at \(X=X_k\).
\end{proof}

\begin{rmk}[What is sharp and what is not]\label{rem:rate-constants}
The sufficient decrease estimate is sharp, but the rate \(\rho_p\) is
uniform over the spectral box containing the iterates after the first step.
The determinant bound for \(\lambda_-\) can make this rate deteriorate rapidly
with the dimension.  On random instances with \(n\le5\), the certified
\(\rho_p\) goes down to \(10^{-57}\), against a per-step contraction of order
one.  Theorem~\ref{thm:linear-convergence} thus establishes global linear
convergence with explicit constants depending on the data, weights, \(p\),
and \(X_0\), but the bound can be too small to predict the observed rate.
A local estimate from the derivative of \(\Tc\) at \(X_\star\) may be sharper.
The a posteriori bound \eqref{eq:posteriori} inherits the same limitation;
Proposition~\ref{prop:certificate} instead bounds \(\min\Fc_w\) from below
without using \(\lambda_\pm\).
\end{rmk}

\FloatBarrier
\begin{aiusage}  %
The results in this paper were developed through extensive use of an AI
assistant (OpenAI's GPT~Pro). None of the claims were obtained zero-shot:
they emerged through sustained back-and-forth discussion, including failed
attempts, counterexamples, audits, and repairs.  The AI-driven side of this
work began in October 2025; the first complete proof was obtained on
11~August 2026.
\end{aiusage}

\aiacknowledgements

\appendix

\section*{Appendices}
\begin{center}\small
\begin{tabular}{@{}l l r@{}}
\ref{sec:balanced}      & \nameref{sec:balanced}      & \pageref{sec:balanced}\\
\ref{sec:endpoints}     & \nameref{sec:endpoints}     & \pageref{sec:endpoints}\\
\ref{sec:variance}      & \nameref{sec:variance}      & \pageref{sec:variance}\\
\ref{app:convergence}   & \nameref{app:convergence}   & \pageref{app:convergence}\\
\ref{app:first-order}   & \nameref{app:first-order}   & \pageref{app:first-order}
\end{tabular}
\end{center}

\section{A balanced trace inequality}\label{sec:balanced}

For comparison with the support inequality, we record the trace estimate
obtained by combining Theorem~\ref{thm:partition} with the
Araki--Lieb--Thirring inequality.  For positive definite \(P,Q\), the
parallel sum is \(P\parsum Q:=(P^{-1}+Q^{-1})^{-1}=P(P+Q)^{-1}Q\).
It extends to semidefinite matrices by positive regularization
\citep{anderson1969series} and satisfies
\((\alpha P)\parsum(\alpha Q)=\alpha(P\parsum Q)\) for \(\alpha\ge0\).

\begin{theorem}[Balanced trace inequality]\label{thm:balanced-trace}
Let \(1\le p\le2\), \(t=p/2\), and let \(A,B,C\succ0\).  Define
\(Y:=A^{1/2}BA^{1/2}\) and \(Z:=A^{1/2}CA^{1/2}\).  If
\(Y^t+Z^t=\lambda A^p\) for some \(\lambda>0\), then
\begin{equation}\label{eq:balanced-trace}
  \Phi_p(B,C)\ge\lambda\,\trace(Y^t\parsum Z^t).
\end{equation}
The conclusion also holds when \(A,B,C\) have a common support, by restricting
to that support.
\end{theorem}

\begin{proof}
Set \(P:=Y^t\), \(Q:=Z^t\), and \(T:=P+Q=\lambda A^p\).
Since \(t/p=1/2\), we have \(P^{1/p}=Y^{1/2}\), \(Q^{1/p}=Z^{1/2}\),
and \(T^{-1/p}=\lambda^{-1/p}A^{-1}\).
Consequently,
\begin{equation}\label{eq:cross-factorization}
  \Phi_p(B,C)
  =\lambda\pnorm{Q^{1/p}T^{-1/p}P^{1/p}}{p}^p.
\end{equation}
Indeed, \(\pnorm{Z^{1/2}A^{-1}Y^{1/2}}{p}^p=\Phi_p(B,C)\) by the argument
given for \eqref{eq:cross-exact} in the proof of
Proposition~\ref{prop:support}, with \(A\) in place of \(A_*\), and
\eqref{eq:cross-factorization} follows from the scalar factor
\(\lambda^{-1/p}\).

Apply Theorem~\ref{thm:partition} to the decomposition \(T=Q+P\) and to
\(X=P^{1/p}\).  We get
\begin{equation}\label{eq:partition-balanced}
  \pnorm{Q^{1/p}T^{-1/p}P^{1/p}}{p}^p
  +
  \pnorm{P^{1/p}T^{-1/p}P^{1/p}}{p}^p
  \ge\trace P.
\end{equation}
For \(p\ge1\), the Araki--Lieb--Thirring trace inequality
\citep{araki1990inequality,bhatia1997matrix} gives
\begin{equation}\label{eq:alt-balanced}
  \trace\left(P^{1/p}T^{-1/p}P^{1/p}\right)^p
  \le\trace(PT^{-1}P).
\end{equation}
Indeed, apply the Araki--Lieb--Thirring inequality to \(X=P^{2/p}\),
\(W=T^{-1/p}\), and exponent \(p\).
Combining \eqref{eq:partition-balanced} and \eqref{eq:alt-balanced},
\begin{align}
  \pnorm{Q^{1/p}T^{-1/p}P^{1/p}}{p}^p
  &\ge\trace P-\trace(PT^{-1}P)\notag\\
  &=\trace(PT^{-1}Q)=\trace(P\parsum Q).
  \label{eq:parallel-lower}
\end{align}
Multiply by \(\lambda\) and use \eqref{eq:cross-factorization}.
\end{proof}

\begin{rmk}\label{rem:no-alt-support}
No Araki--Lieb--Thirring inequality is used in
Proposition~\ref{prop:support}.  The partition inequality evaluates the cross
term and the relevant self term simultaneously, which is stronger than first
bounding the self term by a separate trace inequality.
\end{rmk}

\section{Endpoints and closed forms}\label{sec:endpoints}

\begin{corr}[\(p=2\): the arithmetic mean]\label{cor:p2}
For \(p=2\) the unique barycenter is \(X_\star=\sum_iw_iA_i\).
\end{corr}

\begin{proof}
At \(p=2\) the distance \(D_2(A,B)=\pnorm{A-B}{2}\) is Euclidean, so
\(\Fc_w(X)=\pnorm{X-\sum_iw_iA_i}{2}^2+\text{const}\).  The arithmetic mean
lies in \(\pp_n\) and is the unique minimizer by Theorem~\ref{thm:existence}.
\end{proof}

\begin{corr}[\(p=1\): the Bures--Wasserstein barycenter]\label{cor:p1}
For \(p=1\), a positive definite matrix \(X\) is a barycenter if and only if
\begin{equation}\label{eq:bw-euler}
  X=\sum_{i=1}^mw_i\bigl(X^{1/2}A_iX^{1/2}\bigr)^{1/2}.
\end{equation}
For real symmetric matrices, this is the covariance equation for a barycenter of
the centred Gaussians \(N(0,A_i)\) in \(2\)-Wasserstein space
\citep{agueh2011barycenters,alvarez2016fixed}.
\end{corr}

\begin{proof}
The characterization is Theorem~\ref{thm:euler} at \(p=1\).  For real
symmetric covariances, \(D_1\) is the \(W_2\) distance between the
corresponding centred Gaussians \citep{olkin1982distance,bhatia2019bures}.
\end{proof}

For commuting inputs, the barycenter formula is explicit at every exponent.

\begin{prop}[Commuting inputs]\label{prop:commuting}
If \(A_1,\dots,A_m\) commute pairwise and \(\sum_iw_iA_i^{p/2}\succ0\), then
\begin{equation}\label{eq:power-mean}
  X_\star=\Bigl(\sum_{i=1}^mw_iA_i^{p/2}\Bigr)^{2/p}
\end{equation}
is a barycenter, the power mean of order \(p/2\).
On commuting matrices, the map \(A\mapsto A^{p/2}\) converts \(D_p\) to the
Hilbert--Schmidt distance:
\begin{equation}\label{eq:commuting-isometry}
  D_p(A,B)=\pnorm{A^{p/2}-B^{p/2}}{2}.
\end{equation}
\end{prop}

\begin{proof}
Let \(X\) be the right side of \eqref{eq:power-mean}; it commutes with every
\(A_i\) and is positive definite.  Then
\(\bigl(X^{1/2}A_iX^{1/2}\bigr)^{p/2}=(XA_i)^{p/2}=X^{p/2}A_i^{p/2}\), whence
\[
  \sum_iw_i\bigl(X^{1/2}A_iX^{1/2}\bigr)^{p/2}
  =X^{p/2}\sum_iw_iA_i^{p/2}=X^{p/2}X^{p/2}=X^p,
\]
which is \eqref{eq:euler}; apply Theorem~\ref{thm:euler}.  For
\eqref{eq:commuting-isometry}, commutativity gives
\(\Phi_p(A,B)=\trace(AB)^{p/2}=\trace(A^{p/2}B^{p/2})\), so
\(D_p(A,B)^2=\trace(A^{p/2}-B^{p/2})^2\).
\end{proof}

\begin{rmk}[Equivariance]\label{rem:equivariance}
Since \(\Phi_p(UAU^*,UBU^*)=\Phi_p(A,B)\) for unitary \(U\), barycenters are
unitarily equivariant: if \(X_\star\) is a barycenter, then
\(UX_\star U^*\) is a barycenter of \((UA_iU^*,w_i)\).
Since \(D_p(sA,sB)^2=s^pD_p(A,B)^2\) for \(s\ge0\),
\(sX_\star\) is likewise a barycenter of \((sA_i,w_i)\).
General invertible congruences need not preserve barycenters, except at
\(p=2\), where the arithmetic mean is equivariant.
\end{rmk}

\begin{rmk}[Endpoint checks]\label{rem:lin-endpoints}
For the matrices \(A_i=X_iX_i^*\) in the proof of
Theorem~\ref{thm:lin-block-norm}, at \(p=2\) we have
\(\pnorm{H_{ij}}{2}^2=\trace(A_iA_j)\), so the matrix in
\eqref{eq:lin-block-norm-psd} is the ordinary
Hilbert--Schmidt Gram matrix of \(A_1,A_2,A_3\).  Its positivity for any
number of blocks is the Marcus--Watkins case.  At \(p=1\), three-point
positivity for the Bures affinity recovers Drury's theorem.
\end{rmk}

\section{The variance inequality and three-point geometry}\label{sec:variance}

To bound \(\min\Fc_w\), we optimize the radius and apply a support inequality
on \(\sphere\).

\subsection{Reduction to the sphere}\label{sec:reduction}

\begin{lemma}[Radial reduction]\label{lem:radial-reduction}
For \(1\le p\le2\) and \(w\in\simplex\),
\(\Psi_*:=\max_{\what X\in\sphere}\Psi_w(\what X)\ge0\), and
\begin{equation}\label{eq:min-frechet}
  \min_{X\in\pp_n}\Fc_w(X)=\sum_{i=1}^mw_i\trace A_i^p-\Psi_*^2 .
\end{equation}
Moreover \(X\) is a barycenter if and only if \(X=\Psi_*^{2/p}\what X\) for some
maximizer \(\what X\) of \(\Psi_w\) on \(\sphere\).
\end{lemma}

\begin{proof}
Every nonzero \(X\in\pp_n\) is uniquely \(X=s\what X\) with \(\what X\in\sphere\)
and \(s=r_p(X)^{2/p}>0\); the value \(s=0\) corresponds to \(X=0\).  By
\eqref{eq:frechet-expansion} and homogeneity,
\[
  \Fc_w(s\what X)
  =s^p-2s^{p/2}\Psi_w(\what X)+\sum_iw_i\trace A_i^p .
\]
Substituting \(u:=s^{p/2}\in[0,\infty)\), the dependence on \(s\) is
\(u^2-2u\Psi_w(\what X)\), which is minimized uniquely at
\(u=\Psi_w(\what X)\) with value \(-\Psi_w(\what X)^2\).  Since
\(\Psi\mapsto-\Psi^2\) is strictly decreasing on \([0,\infty)\), minimizing
over \(\what X\) selects exactly the maximizers of \(\Psi_w\) on \(\sphere\),
and gives \eqref{eq:min-frechet}.

The value \(\Psi_*\) is positive if \(w_i>0\) and \(A_i\ne0\) for some \(i\): at
\(\what X=n^{-1/p}I\), the corresponding term is
\(\Phi_p(\what X,A_i)=n^{-1/2}\trace A_i^{p/2}>0\).
If no such index exists, \(\Psi_w\equiv0\).
\end{proof}

\begin{corr}[Value and radius of the barycenter]\label{cor:value}
Let \(X_\star\succ0\) be a barycenter.  Then
\(\trace X_\star^p=\Psi_w(X_\star)=\Psi_*^2\),
\begin{align}
  \min_{X\in\pp_n}\Fc_w(X)&=\sum_{i=1}^mw_i\trace A_i^p-\trace X_\star^p,
  \label{eq:min-value}\\
  r_p(X_\star)&\le\sum_{i=1}^mw_ir_p(A_i).
  \label{eq:radial-bound}
\end{align}
\end{corr}

\begin{proof}
Taking traces in \eqref{eq:euler} gives
\(\trace X_\star^p=\sum_iw_i\Phi_p(X_\star,A_i)=\Psi_w(X_\star)\); the value
\(\Psi_*^2\) and \eqref{eq:min-value} then follow from
Lemma~\ref{lem:radial-reduction} and \eqref{eq:frechet-expansion}.  For
\eqref{eq:radial-bound}, combine \(r_p(X_\star)^2=\trace X_\star^p=\Psi_w(X_\star)\)
with \eqref{eq:holder-affinity} to obtain
\(r_p(X_\star)^2=\sum_iw_i\Phi_p(X_\star,A_i)
\le r_p(X_\star)\sum_iw_ir_p(A_i)\), and divide by \(r_p(X_\star)>0\).
\end{proof}

\subsection{A multipoint support inequality and the variance inequality}

To control \(\Psi_*\) for all the data points at once, we extend
Proposition~\ref{prop:support} to a weighted sum of affinities.

\begin{prop}[Multipoint support inequality]\label{prop:support-multi}
Let \(1\le p\le2\), let \(A,A_1,\dots,A_m\in\sphere\), and let
\(a_1,\dots,a_m\ge0\).  Then
\begin{equation}\label{eq:support-multi}
  \sum_{i=1}^ma_i\Phi_p(A,A_i)
  \le
  \Bigl(\sum_{i=1}^ma_i^2
    +\sum_{i\ne j}a_ia_j\Phi_p(A_i,A_j)\Bigr)^{1/2}.
\end{equation}
\end{prop}

\begin{proof}
Assume first that every \(a_i>0\) and every \(A_i\succ0\).  Let \(A_*\)
maximize \(F(X)=\sum_ia_i\Phi_p(X,A_i)\) over \(\sphere\); by
Lemma~\ref{lem:positive-maximizer}, \(A_*\succ0\).  Put
\(Y_i:=A_*^{1/2}A_iA_*^{1/2}\).  By Lemma~\ref{lem:differential} the Lagrange
condition is \(\sum_{i=1}^ma_iY_i^t=\lambda A_*^p\),
and taking traces with \(\trace A_*^p=1\) identifies \(\lambda=F(A_*)\).  Set
\(P_i:=a_iY_i^t\) and \(T:=\sum_iP_i=\lambda A_*^p\).
Since \(t/p=1/2\), we have \(P_i^{1/p}=a_i^{1/p}Y_i^{1/2}\) and
\(T^{-1/p}=\lambda^{-1/p}A_*^{-1}\), so exactly as in
\eqref{eq:cross-exact}--\eqref{eq:self-Q-exact},
\begin{equation}\label{eq:multi-evaluations}
  \pnorm{P_j^{1/p}T^{-1/p}P_i^{1/p}}{p}^p
  =\begin{cases}
    \dfrac{a_ia_j}{\lambda}\Phi_p(A_i,A_j), & i\ne j,\\[2mm]
    \dfrac{a_i^2}{\lambda}, & i=j,
  \end{cases}
\end{equation}
the diagonal case using \(\trace A_i^p=1\).  Applying
Theorem~\ref{thm:partition} to the decomposition \(T=\sum_jP_j\) with
\(X=P_i^{1/p}\) gives
\[
  \sum_{j=1}^m\pnorm{P_j^{1/p}T^{-1/p}P_i^{1/p}}{p}^p
  \ \ge\ \trace P_i=a_i\Phi_p(A_*,A_i).
\]
Summing over \(i\) and inserting \eqref{eq:multi-evaluations},
\[
  \frac1\lambda\Bigl(\sum_ia_i^2+\sum_{i\ne j}a_ia_j\Phi_p(A_i,A_j)\Bigr)
  \ \ge\ \sum_ia_i\Phi_p(A_*,A_i)=\lambda .
\]
Since \(F(A)\le\lambda\), this proves \eqref{eq:support-multi}.
Approximate singular \(A_i\) and omit terms with \(a_i=0\).
\end{proof}

\begin{theorem}[Variance inequality]\label{thm:variance}
Let \(1\le p\le2\), \(A_1,\dots,A_m\in\pp_n\) and \(w\in\simplex\).  Then
\begin{equation}\label{eq:variance}
  \min_{X\in\pp_n}\ \sum_{i=1}^mw_iD_p(X,A_i)^2
  \ \ge\
  \sum_{1\le i<j\le m}w_iw_jD_p(A_i,A_j)^2 .
\end{equation}
Equality holds when \(p=2\), and on any commuting family.
\end{theorem}

\begin{proof}
Set \(r_i:=r_p(A_i)\).  For \(r_i>0\), let \(\what A_i\) be the normalization
\eqref{eq:normalization}; for \(r_i=0\), choose any \(\what A_i\in\sphere\).
Put \(z_{ij}:=\Phi_p(\what A_i,\what A_j)\).  By
homogeneity \eqref{eq:normalization}, for \(\what X\in\sphere\),
\(\Psi_w(\what X)=\sum_iw_ir_i\Phi_p(\what X,\what A_i)\),
so Proposition~\ref{prop:support-multi} applied with \(a_i=w_ir_i\) bounds the
maximum \(\Psi_*\):
\[
  \Psi_*^2\ \le\ \sum_iw_i^2r_i^2+\sum_{i\ne j}w_iw_jr_ir_jz_{ij}.
\]
Insert this into \eqref{eq:min-frechet} and use \(\trace A_i^p=r_i^2\):
\[
  \min_X\Fc_w(X)
  \ \ge\ \sum_iw_ir_i^2-\sum_iw_i^2r_i^2-\sum_{i\ne j}w_iw_jr_ir_jz_{ij}
  =\sum_iw_i(1-w_i)r_i^2-\sum_{i\ne j}w_iw_jr_ir_jz_{ij}.
\]
Since \(1-w_i=\sum_{j\ne i}w_j\), the first sum equals
\(\sum_{i\ne j}w_iw_jr_i^2=\sum_{i<j}w_iw_j(r_i^2+r_j^2)\), while the second
equals \(2\sum_{i<j}w_iw_jr_ir_jz_{ij}\).  Hence
\[
  \min_X\Fc_w(X)
  \ \ge\ \sum_{i<j}w_iw_j\bigl(r_i^2+r_j^2-2r_ir_jz_{ij}\bigr)
  =\sum_{i<j}w_iw_jD_p(A_i,A_j)^2
\]
by \eqref{eq:cone-law-cosines}.  At \(p=2\) the distance is Euclidean and
\eqref{eq:variance} is the classical identity
\(\sum_iw_i\pnorm{X-A_i}{2}^2=\pnorm{X-\sum_jw_jA_j}{2}^2
+\sum_{i<j}w_iw_j\pnorm{A_i-A_j}{2}^2\) evaluated at the mean; on a commuting
family the same identity applies after the isometry
\eqref{eq:commuting-isometry}.
\end{proof}

\begin{rmk}[Two points and geodesicity]\label{rem:midpoints}
For \(m=2\) inequality \eqref{eq:variance} reads
\begin{equation}\label{eq:two-point-variance}
  \min_{X\in\pp_n}\bigl(w_1D_p(X,A_1)^2+w_2D_p(X,A_2)^2\bigr)
  \ \ge\ w_1w_2D_p(A_1,A_2)^2 .
\end{equation}
The reverse inequality holds in \emph{any} geodesic metric space: the point of
a geodesic from \(A_1\) to \(A_2\) at parameter \(w_2\) attains the value
\(w_1w_2D_p(A_1,A_2)^2\).  Strict inequality therefore rules out a geodesic
joining that pair; for \(w_1=w_2=1/2\), it also rules out a midpoint.

Equality holds at both endpoints: \(D_2\) is Euclidean, and the
Bures--Wasserstein space at \(p=1\) is geodesic
\citep{takatsu2011wasserstein,bhatia2019bures}.  For \(1<p<2\),
\eqref{eq:two-point-variance} with \(w_1=w_2=1/2\) is strict for suitable
pairs of rank-one projections in dimension two (we omit the proof).
Consequently, \(D_p\) is not a length metric on \(\pp_n\) for any \(n\ge2\),
even on the positive definite cone; in particular, \((\pp_n,D_p)\) is not
geodesic.  Commuting pairs do have midpoints, by \eqref{eq:commuting-isometry}.
\end{rmk}

\subsection{An a posteriori optimality certificate}\label{sec:computation}

The concavity of \(\Psi_w\) gives a lower bound on the minimum of \(\Fc_w\).
Once \(\nabla\Psi_w\) has been computed, the bound requires one Schatten norm.

\begin{prop}[Optimality certificate]\label{prop:certificate}
Let \(1\le p\le2\), \(w\in\simplex\), and \(X_0\succ0\).  Put
\begin{equation}\label{eq:certificate-gradient}
  G_0:=\tfrac p2\,X_0^{-1/2}
  \Bigl(\sum_iw_i\bigl(X_0^{1/2}A_iX_0^{1/2}\bigr)^{p/2}\Bigr)X_0^{-1/2}
  \ \succeq0 .
\end{equation}
Then, with \(\Psi_*\) as in Lemma~\ref{lem:radial-reduction},
\begin{equation}\label{eq:certificate}
  \Psi_*\ \le\ \Bigl(1-\tfrac p2\Bigr)\Psi_w(X_0)+\pnorm{G_0}{p'},
\end{equation}
and consequently \(\min_X\Fc_w(X)\) is at least
\(\sum_iw_i\trace A_i^p\) minus the square of the right side of
\eqref{eq:certificate}.  The bound \eqref{eq:certificate} holds with equality
when \(X_0\) is a maximizer of \(\Psi_w\) on \(\sphere\).
\end{prop}

\begin{proof}
Let \(\Kc:=\{X\in\pp_n:\trace X^p\le1\}\), a convex set by
Proposition~\ref{prop:convexity}(ii).  By homogeneity of \(\Psi_w\),
\(\max_{\Kc}\Psi_w=\max_{\sphere}\Psi_w=\Psi_*\).  Since \(\Psi_w\) is concave
(Proposition~\ref{prop:convexity}(i)) and differentiable at \(X_0\succ0\) with
\(\nabla\Psi_w(X_0)=G_0\) (Lemma~\ref{lem:differential}),
\[
  \Psi_*=\max_{X\in\Kc}\Psi_w(X)
  \le\Psi_w(X_0)-\ip{G_0}{X_0}+\max_{X\in\Kc}\ip{G_0}{X}.
\]
The inner product term simplifies: writing \(S_0\) for the bracket in
\eqref{eq:certificate-gradient},
\[
  \ip{G_0}{X_0}=\tfrac p2\trace(X_0^{-1/2}S_0X_0^{-1/2}X_0)=\tfrac p2\trace
  S_0=\tfrac p2\Psi_w(X_0).
\]
For the last term, H\"older gives
\(\ip{G_0}{X}\le\pnorm{G_0}{p'}\pnorm Xp\le\pnorm{G_0}{p'}\)
on \(\Kc\), including \(p=1\), when \(p'=\infty\).
Substituting these bounds into the concavity inequality proves
\eqref{eq:certificate}; combining with
\eqref{eq:min-frechet} gives the stated lower bound on \(\min\Fc_w\).

For the equality case, let \(X_0\in\sphere\) maximize \(\Psi_w\).  Since
\(X_0\succ0\), Lemma~\ref{lem:differential} gives the Lagrange condition
\(\sum_iw_i(X_0^{1/2}A_iX_0^{1/2})^{p/2}=\Psi_*X_0^p\), so
\(G_0=\tfrac p2\Psi_*X_0^{p-1}\).  The normalization
\(\trace X_0^p=1\) gives \(\pnorm{X_0^{p-1}}{p'}=1\), also at \(p=1\),
where \(X_0^0=I\) and \(p'=\infty\).  Hence
\(\pnorm{G_0}{p'}=\tfrac p2\Psi_*\).
Since \(\Psi_w(X_0)=\Psi_*\), the right side of \eqref{eq:certificate} equals
\((1-\tfrac p2)\Psi_*+\tfrac p2\Psi_*=\Psi_*\).
\end{proof}

\subsection{Three-point consequences}

The block positivity theorem also controls quadratic forms in the squared
distances of a triple.

\begin{corr}[Three-point negative type]\label{cor:p-bures-three-point-cnd}
Let \(1\le p\le2\).  For every \(A_1,A_2,A_3\in\pp_n\), the \(3\times3\)
affinity matrix \(K:=[\Phi_p(A_i,A_j)]_{i,j=1}^3\) is positive semidefinite.
Consequently, the squared distance matrix
\(\Delta:=[D_p(A_i,A_j)^2]_{i,j=1}^3\) is conditionally negative semidefinite:
\(c^{\mathsf T}\Delta c\le0\) for every
\(c=(c_1,c_2,c_3)^{\mathsf T}\in\reals^3\) with \(c_1+c_2+c_3=0\).
\end{corr}

\begin{proof}
Apply Theorem~\ref{thm:lin-block-norm} to the positive semidefinite block Gram
matrix with \(H_{ij}:=A_i^{1/2}A_j^{1/2}\) to obtain \(K\succeq0\).  This
is also the matrix case of \citet[Theorem~1.1]{vuong2026three}.  For
\(c_1+c_2+c_3=0\), the radial terms in the definition of \(D_p^2\) cancel, and
hence
\begin{align}
  c^{\mathsf T}\Delta c
  &=\sum_{i,j=1}^3c_ic_j
  \left(\trace A_i^p+\trace A_j^p-2\Phi_p(A_i,A_j)\right)\notag\\*
  &=-2c^{\mathsf T}Kc\le0.
  \label{eq:three-point-cnd-proof}
\end{align}
\end{proof}

\begin{corr}[A cosine-type support law]\label{cor:cosine-law}
For \(A,B,C\in\sphere\), write \(x=\Phi_p(A,B)\), \(y=\Phi_p(A,C)\),
and \(z=\Phi_p(B,C)\).  Then
\begin{equation}\label{eq:three-point-cosine}
  z\ge xy-\sqrt{(1-x^2)(1-y^2)}.
\end{equation}
\end{corr}

\begin{proof}
Inequality \eqref{eq:three-point-cosine} is equivalent to the angular triangle inequality after writing
\(x=\cos\alpha\), \(y=\cos\beta\), and \(z=\cos\gamma\).
\end{proof}

\section{Proofs of the convergence estimates, and the endpoint \texorpdfstring{\(p=1\)}{p=1}}\label{app:convergence}

\begin{proof}[\textbf{Proof of Lemma~\ref{lem:spectral-confinement}}]
By Lemma~\ref{lem:sufficient-decrease}, \(\Fc_w(X_k)\le F_0\) for every \(k\).
By \eqref{eq:holder-affinity}, \(D_p(X,A_i)^2\ge(r_p(X)-r_p(A_i))^2\), so
Jensen gives
\[
  \Fc_w(X)\ \ge\ \sum_iw_i\bigl(r_p(X)-r_p(A_i)\bigr)^2
  \ \ge\ \bigl(r_p(X)-\bar r\bigr)^2 .
\]
Hence \(r_p(X_k)\le\bar r+\sqrt{F_0}\), and since
\(\lambda_{\max}(X_k)^p\le\trace X_k^p=r_p(X_k)^2\), also
\(\lambda_{\max}(X_k)\le\lambda_+\).

For the lower bound, \(A_i\succeq\lambda_{\min}(A_i)I\) gives
\(X^{1/2}A_iX^{1/2}\succeq\lambda_{\min}(A_i)X\), and \(0<t\le1\), so the
L\"owner--Heinz inequality \citep{bhatia1997matrix} yields
\(\bigl(X^{1/2}A_iX^{1/2}\bigr)^t\succeq\lambda_{\min}(A_i)^tX^t\) and
therefore \(S(X)\succeq c_AX^t\).  Taking determinants gives
\(\det S(X)\ge c_A^n(\det X)^t\), whence
\begin{equation}\label{eq:det-T}
  \det\Tc(X)=(\det X)^{-1}\bigl(\det S(X)\bigr)^{1/t}\ \ge\ c_A^{\,n/t}.
\end{equation}
So \(\det X_k\ge c_A^{n/t}\) for \(k\ge1\), and combining with the upper bound,
\[
  \lambda_{\min}(X_k)\ \ge\ \frac{\det X_k}{\lambda_{\max}(X_k)^{n-1}}
  \ \ge\ \frac{c_A^{\,n/t}}{\lambda_+^{\,n-1}}=\lambda_- .
\]
A barycenter \(X_\star\) satisfies \(\Fc_w(X_\star)\le F_0\), hence
\(X_\star\preceq\lambda_+I\) by the first paragraph; it is positive definite by
Lemma~\ref{lem:positive-barycenter} and fixed by \(\Tc\) by
Theorem~\ref{thm:euler}, so \eqref{eq:det-T} gives
\(\det X_\star\ge c_A^{n/t}\) and hence \(X_\star\succeq\lambda_-I\).
\end{proof}

\begin{proof}[\textbf{Proof of Lemma~\ref{lem:residual-error-bound}}]
\emph{\textbf{(i) Strong convexity.}}  By \eqref{eq:frechet-expansion} and the symmetry
of \(\Phi_p\),
\begin{equation}\label{eq:F-decomp}
  \Fc_w(X)=\trace X^p-2\sum_{i=1}^mw_i\trace\bigl(A_i^{1/2}XA_i^{1/2}\bigr)^t
  +\text{const},
\end{equation}
and the second term is convex by Proposition~\ref{prop:convexity}(i).  For
\(g(X)=\trace f(X)\) with \(f\) smooth, the Daleckii--Krein formula
\citep{bhatia1997matrix} gives
\(D^2g(X)[H,H]=\sum_{j,k}(f')^{[1]}(\mu_j,\mu_k)|H_{jk}|^2\) in an eigenbasis
of \(X\), where \((f')^{[1]}\) is the first divided difference of \(f'\) and
\(\mu_j\) are the eigenvalues.  With \(f(x)=x^p\) each divided difference is a
value of \(f''(x)=p(p-1)x^{p-2}\) on \([\lambda_-,\lambda_+]\), and \(p\le2\)
makes \(x^{p-2}\) nonincreasing, so on \(\Bc\)
\begin{equation}\label{eq:hessian-trXp}
  D^2\trace X^p[H,H]\ \ge\ p(p-1)\lambda_+^{p-2}\frob{H}^2 .
\end{equation}
Thus \(\Fc_w\) is \(\mu_F\)-strongly convex on the convex set \(\Bc\), with
\(\mu_F:=p(p-1)\lambda_+^{p-2}\).  Minimizing the quadratic lower bound at
\(X\) and evaluating at \(X_\star\in\Bc\) gives
\begin{equation}\label{eq:grad-gap}
  \frob{\nabla\Fc_w(X)}^2\ \ge\ 2\mu_F\bigl(\Fc_w(X)-\Fc_w(X_\star)\bigr).
\end{equation}
The same computation applied to \(U\mapsto D_p(U,V)^2=\trace U^p
-2\Phi_p(U,V)+\trace V^p\), which vanishes together with its gradient at
\(U=V\), gives for \(U,V\in\Bc\)
\begin{equation}\label{eq:Dp-Frobenius-lower}
  D_p(U,V)^2\ \ge\ \frac{\mu_D}{2}\frob{U-V}^2,
  \qquad
  \mu_D:=p(p-1)\lambda_+^{p-2}.
\end{equation}

\emph{\textbf{(ii) The step controls the gradient.}}  Writing
\(X^{p-1}=X^{-1/2}X^pX^{-1/2}\) in \eqref{eq:frechet-gradient},
\begin{equation}\label{eq:gradient-residual-form}
  \nabla\Fc_w(X)=pX^{-1/2}\bigl(X^p-S(X)\bigr)X^{-1/2}.
\end{equation}
Put \(Q:=S(X)^{1/t}=X^{1/2}\Tc(X)X^{1/2}\).  Since \(p=2t\),
\(X^p-S(X)=(X^2)^t-Q^t\), and \(X,\Tc(X)\succeq\lambda_-I\) force
\(X^2\succeq\lambda_-^2I\) and
\(Q\succeq\lambda_-X\succeq\lambda_-^2I\).  On \([\lambda_-^2,\infty)\) the
divided differences of \(x\mapsto x^t\) are bounded by \(t\lambda_-^{2t-2}\),
so the Daleckii--Krein formula gives the Frobenius Lipschitz bound
\(\frob{U^t-V^t}\le t\lambda_-^{2t-2}\frob{U-V}\).  Applying it to \(U=X^2\),
\(V=Q\) and using \(X^2-Q=X^{1/2}\bigl(X-\Tc(X)\bigr)X^{1/2}\),
\[
  \frob{X^p-S(X)}\ \le\ t\lambda_+\lambda_-^{p-2}\frob{X-\Tc(X)} .
\]
Feeding this into \eqref{eq:gradient-residual-form} and using
\(\frob{X^{-1/2}ZX^{-1/2}}\le\lambda_-^{-1}\frob{Z}\),
\begin{equation}\label{eq:gradient-residual}
  \frob{\nabla\Fc_w(X)}\ \le\ C_R\frob{X-\Tc(X)},
  \qquad
  C_R:=pt\,\lambda_+\lambda_-^{p-3}=\frac{p^2}{2}\lambda_+\lambda_-^{p-3}.
\end{equation}

\emph{\textbf{(iii) Combination.}}  Chaining \eqref{eq:Dp-Frobenius-lower},
\eqref{eq:gradient-residual} and \eqref{eq:grad-gap},
\[
  D_p\bigl(X,\Tc(X)\bigr)^2
  \ \ge\ \frac{\mu_D}{2}\frob{X-\Tc(X)}^2
  \ \ge\ \frac{\mu_D}{2C_R^2}\frob{\nabla\Fc_w(X)}^2
  \ \ge\ \frac{\mu_D\mu_F}{C_R^2}\bigl(\Fc_w(X)-\Fc_w(X_\star)\bigr),
\]
and substituting the three constants,
\[
  \frac{\mu_D\mu_F}{C_R^2}
  =\frac{p^2(p-1)^2\lambda_+^{2p-4}}
        {\tfrac{p^4}4\lambda_+^2\lambda_-^{2p-6}}
  =\frac{4(p-1)^2}{p^2}\left(\frac{\lambda_-}{\lambda_+}\right)^{6-2p}
  =\rho_p .
\]
Finally \(4(p-1)^2/p^2\le1\) exactly because \(p\le2\), and
\(\lambda_-\le\lambda_+\) with \(6-2p>0\), so \(\rho_p\le1\).
\end{proof}

At \(p=1\) the leading term \(\trace X^p\) is linear, so
\eqref{eq:hessian-trXp} degenerates.  For positive definite data, the fidelity
terms supply the curvature needed for the error bound.

\begin{prop}[The Bures endpoint]\label{prop:bures-endpoint-rate}
Let \(p=1\), \(A_1,\dots,A_m\succ0\), and \(w\in\simplex\) with all \(w_i>0\).
Write \(\alpha_i:=\lambda_{\min}(A_i)\) and \(\beta_i:=\lambda_{\max}(A_i)\),
and set
\begin{equation}\label{eq:p1-constants}
  \mu_F^{(1)}:=\sum_{i=1}^mw_i\frac{\alpha_i^2}{2(\beta_i\lambda_+)^{3/2}},
  \qquad
  \mu_D^{(1)}:=\frac{\lambda_-^2}{2\lambda_+^3},
  \qquad
  C_R^{(1)}:=\frac{\lambda_+}{2\lambda_-^2}.
\end{equation}
Then Lemma~\ref{lem:residual-error-bound} and
Theorem~\ref{thm:linear-convergence} hold verbatim at \(p=1\) with
\begin{equation}\label{eq:rho-p1}
  \rho_1:=\min\left\{1,\
  \frac{\mu_F^{(1)}\mu_D^{(1)}}{\bigl(C_R^{(1)}\bigr)^2}\right\}>0
\end{equation}
in place of \(\rho_p\), and with \(\mu_F^{(1)}\) in place of
\(p(p-1)\lambda_+^{p-2}\) in \eqref{eq:linear-iterates}.
\end{prop}

\begin{proof}
Only part (i) of the proof of Lemma~\ref{lem:residual-error-bound} changes.
For \(f(x)=\sqrt x\),
\[
  -\frac{f'(x)-f'(y)}{x-y}=\frac{1}{2\sqrt x\sqrt y(\sqrt x+\sqrt y)}
  \ \ge\ \frac{1}{4(\beta_i\lambda_+)^{3/2}},
  \qquad 0<x,y\le\beta_i\lambda_+ ,
\]
and the eigenvalues of \(A_i^{1/2}XA_i^{1/2}\) do not exceed
\(\beta_i\lambda_+\) for \(X\in\Bc\).  Hence, by Daleckii--Krein as above and
\(\frob{A_i^{1/2}HA_i^{1/2}}\ge\alpha_i\frob H\),
\[
  D^2\Bigl[-2\trace\bigl(A_i^{1/2}XA_i^{1/2}\bigr)^{1/2}\Bigr][H,H]
  \ \ge\ \frac{\alpha_i^2}{2(\beta_i\lambda_+)^{3/2}}\frob H^2 ,
\]
so \(\Fc_w\) is \(\mu_F^{(1)}\)-strongly convex on \(\Bc\).  Applying the same
bound with the fixed second argument \(V\in\Bc\) in place of \(A_i\), so that
\(\lambda_-\) and \(\lambda_+\) replace \(\alpha_i\) and \(\beta_i\), gives
\eqref{eq:Dp-Frobenius-lower} with \(\mu_D^{(1)}\).  Part (ii) is
unchanged and yields \(C_R^{(1)}=pt\lambda_+\lambda_-^{p-3}\) at \(p=1\),
\(t=\tfrac12\).  Part (iii) then gives \eqref{eq:error-bound} with constant
\(\mu_F^{(1)}\mu_D^{(1)}/(C_R^{(1)})^2\); replacing it by \(\rho_1\) only
weakens the bound, and guarantees \(1-\rho_1\ge0\).  The proof of
Theorem~\ref{thm:linear-convergence} then applies verbatim.
\end{proof}

\begin{rmk}[Uniqueness at \(p=1\)]\label{rem:p1-uniqueness}
Strong convexity of \(\Fc_w\) on \(\Bc\) and the fact that every barycenter
lies in \(\Bc\) (Lemma~\ref{lem:spectral-confinement}) together give uniqueness
of the barycenter at \(p=1\) whenever all \(A_i\succ0\).
This consequence of the rate analysis uses a stronger hypothesis than the
criterion recalled after Theorem~\ref{thm:existence}, which needs only one
nonsingular \(A_i\) with \(w_i>0\).
\end{rmk}

\begin{rmk}[Endpoints]\label{rem:rate-endpoints}
At \(p=2\) the map \(\Tc\) is constant and \(X_1=X_\star\), so
\eqref{eq:linear-rate} holds trivially, and
\eqref{eq:sufficient-decrease} becomes the exact bias--variance decomposition
\(\Fc_w(X)-\Fc_w(\bar A)=\pnorm{X-\bar A}{2}^2\) for
\(\bar A=\sum_iw_iA_i\).  At \(p=1\),
Lemma~\ref{lem:sufficient-decrease} recovers the sufficient decrease estimate
for the Gaussian Wasserstein barycenter iteration in
\citet[Proposition~3.3]{alvarez2016fixed}.
\end{rmk}

\section{A numerical comparison with first-order methods}\label{app:first-order}

Convexity of \(\Fc_w\) on the cone (Proposition~\ref{prop:convexity}) allows
a comparison of \(\Tc\) with gradient descent (GD) and an accelerated FISTA
variant with adaptive restart (AGD), both using backtracking.
Work is measured in \emph{passes}: one sweep over
the \(m\) data matrices requiring an eigendecomposition of each.  A step of
\(\Tc\), one gradient, and one evaluation of \(\Fc_w\) each cost one pass.
We count all objective evaluations required by GD and AGD, including those at
extrapolated points; evaluations used only to monitor \(\Tc\) are not charged.

Table~\ref{tab:solver-comparison} reports the medians and
Figure~\ref{fig:solver-comparison} a single trajectory.  For each instance,
the relative objective gap is
\((\Fc_w(X)-F_{\mathrm{ref}})/\max\{1,|F_{\mathrm{ref}}|\}\), where
\(F_{\mathrm{ref}}\) is the lowest final objective value obtained by the three
methods.  At the tested exponents below \(2\), the median pass counts are \(4\)--\(12\)
for \(\Tc\), \(41\)--\(193\) for GD, and \(45\)--\(178\) for AGD.
Acceleration does not consistently improve on GD.
These comparisons concern the tested implementations, without alternative
parametrizations for GD or AGD.  The convergence guarantee in
Theorem~\ref{thm:linear-convergence} does not explain the observed advantage
of \(\Tc\).

\begin{table}[htbp]
\centering
\caption{Median passes to a relative objective gap of \(10^{-13}\), among
the runs reaching that target out of \(15\) random instances with \(m=4\).
All three methods use the same random positive definite start on each
instance.  At \(p=2\), the map \(\Tc\) is exact in a single step.}
\label{tab:solver-comparison}
\small
\begin{tabular}{@{}c rrr rrr rrr@{}}
\toprule
 & \multicolumn{3}{c}{\(n=2\)}
 & \multicolumn{3}{c}{\(n=5\)}
 & \multicolumn{3}{c}{\(n=10\)}\\
\cmidrule(lr){2-4}\cmidrule(lr){5-7}\cmidrule(lr){8-10}
\(p\) & \(\Tc\) & GD & AGD & \(\Tc\) & GD & AGD & \(\Tc\) & GD & AGD\\
\midrule
  $1.00$ & 6 & 75 & 81 & 10 & 115 & 151 & 12 & 193 & 178 \\
  $1.25$ & 6 & 92 & 60 & 8 & 79 & 98 & 9 & 106 & 129 \\
  $1.50$ & 6 & 57 & 55 & 7 & 65 & 65 & 8 & 59 & 80 \\
  $1.75$ & 4 & 48 & 46 & 5 & 42 & 45 & 6 & 41 & 48 \\
  $2.00$ & 1 & 4 & 4 & 1 & 3 & 36 & 1 & 3 & 4 \\
\bottomrule
\end{tabular}
\end{table}

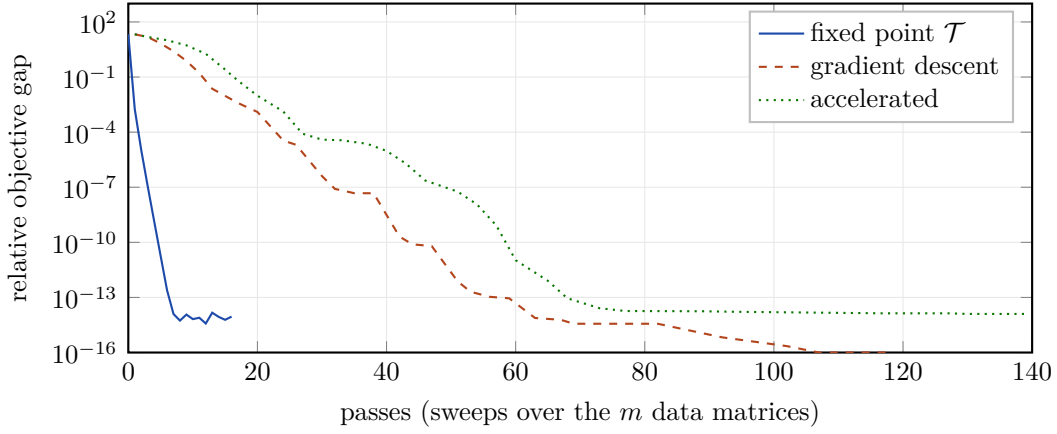
\begin{figure}[htbp]
\centering
\begin{tikzpicture}
\begin{semilogyaxis}[
    width=0.82\linewidth, height=6.2cm,
    xlabel={passes (sweeps over the \(m\) data matrices)},
    ylabel={relative objective gap},
    xmin=0, xmax=140, ymin=1e-16, ymax=1e3,
    ytick={1e-16,1e-13,1e-10,1e-7,1e-4,1e-1,1e2},
    grid=both, grid style={gray!18},
    tick label style={font=\small},
    label style={font=\small},
    legend style={font=\small, at={(0.98,0.98)}, anchor=north east,
                  cells={anchor=west}, draw=gray!50},
    no markers, thick,
]
\addplot[cdarkblue, solid] coordinates {(0,2.1286e+01) (1,1.7871e-03) (2,1.0734e-05) (3,1.1575e-07) (4,1.4422e-09) (5,1.8757e-11) (6,2.4889e-13) (7,1.2523e-14) (8,5.4787e-15) (9,1.1628e-14) (10,6.5968e-15) (11,7.8267e-15) (12,3.8015e-15) (13,1.4759e-14) (14,8.7212e-15) (15,6.0377e-15) (16,9.0566e-15)};
\addplot[cdarkred, dashed] coordinates {(1,2.1286e+01) (3,1.5191e+01) (5,6.7314e+00) (7,2.5355e+00) (9,7.5340e-01) (11,1.6480e-01) (13,2.3017e-02) (17,3.8451e-03) (20,1.2529e-03) (24,3.4970e-05) (26,2.0352e-05) (30,4.0542e-07) (32,8.1124e-08) (35,4.7537e-08) (38,4.6942e-08) (42,2.0613e-10) (44,7.8125e-11) (47,6.2349e-11) (51,7.6321e-13) (53,2.0372e-13) (56,1.0778e-13) (59,8.9671e-14) (63,7.6031e-15) (67,6.0377e-15) (69,3.6897e-15) (82,3.6897e-15) (92,6.7086e-16) (102,2.2362e-16) (107,1.0000e-16) (111,1.0000e-16) (114,1.0000e-16) (118,1.0000e-16)};
\addplot[cmygreen, dotted] coordinates {(1,2.1286e+01) (3,1.5191e+01) (6,1.0117e+01) (9,5.3177e+00) (12,1.8420e+00) (15,2.5496e-01) (17,6.2871e-02) (20,9.6298e-03) (24,1.3313e-03) (27,8.1855e-05) (30,3.9168e-05) (33,3.6162e-05) (37,2.3169e-05) (40,9.5078e-06) (43,1.9349e-06) (46,2.3119e-07) (51,5.7997e-08) (54,1.1737e-08) (57,8.3058e-10) (60,1.0634e-11) (65,7.9978e-13) (68,8.9895e-14) (73,2.5940e-14) (77,1.8225e-14) (89,1.7666e-14) (94,1.6548e-14) (117,1.3641e-14) (124,1.3641e-14) (127,1.3641e-14) (130,1.2635e-14) (133,1.2635e-14) (139,1.2635e-14)};
\legend{fixed point \(\Tc\), gradient descent, accelerated}
\end{semilogyaxis}
\end{tikzpicture}
\caption{One instance, \(n=5\), \(m=4\), \(p=1.5\), from a random positive
definite start.  The fixed point reaches a gap of \(10^{-13}\) in \(7\) passes;
gradient descent needs \(59\) and the accelerated variant \(68\).  Each trace
ends at that method's stopping test.  Values below \(10^{-16}\) are floored;
the tails lie at double-precision accuracy.}
\label{fig:solver-comparison}
\end{figure}

\end{document}